\documentclass[11pt,a4paper]{article}

\usepackage{enumitem}

\usepackage[latin1]{inputenc}
\usepackage[english]{babel}
\usepackage{fancyheadings}
\usepackage{amsmath}
\usepackage{subfigure}
\usepackage{amssymb}
\usepackage{color}

\usepackage[amsmath,thmmarks,standard,thref]{ntheorem}
\usepackage{ae}
\usepackage[T1]{fontenc}
\usepackage{graphicx}
\usepackage{epstopdf}
\newcommand{\be}{\begin{eqnarray}}
\newcommand{\ee}{\end{eqnarray}}
\newcommand{\ba}{\begin{align}}
\newcommand{\ea}{\end{align}}
\newcommand{\bi}{\begin{itemize}}
\newcommand{\ei}{\end{itemize}}

\newcommand{\beq}[1]{\begin{equation} \label{#1}}
\newcommand{\eeq}{\end{equation}}
\newcommand{\beqa}{\begin{eqnarray}}
\newcommand{\eeqa}{\end{eqnarray}}
\newcommand{\bal}{\begin{align}}
\newcommand{\eal}{\end{align}}
\newcommand{\bsub}{\begin{subequations}}
\newcommand{\esub}{\end{subequations}}
\newcommand{\eqlab}[1]{\label{eq:#1}}
\renewcommand{\eqref}[1]{(\ref{eq:#1})}

\newcommand{\figref}[1]{Fig.~\ref{fig:#1}}
\newcommand{\figlab}[1]{\label{fig:#1}}

\newcommand{\seclab}[1]{\label{sec:#1}}

\newcommand{\remref}[1]{Remark~\ref{remark:#1}}
\newcommand{\remlab}[1]{\label{remark:#1}}

\newcommand{\thmref}[1]{Theorem~\ref{theorem:#1}}
\newcommand{\thmlab}[1]{\label{theorem:#1}}
\newcommand{\lemmaref}[1]{Lemma~\ref{lemma:#1}}
\newcommand{\lemmalab}[1]{\label{lemma:#1}}
\newcommand{\propref}[1]{Proposition~\ref{proposition:#1}}
\newcommand{\proplab}[1]{\label{proposition:#1}}

\renewcommand{\headheight}{47.0pt}

\usepackage{url}
\newcommand\response[1]{{\color{black}{#1}}}

\usepackage[sort&compress,numbers]{natbib}
\usepackage{authblk}
\title{Mixed-mode oscillations in coupled FitzHugh-Nagumo oscillators: \\ blow-up analysis of cusped singularities}
\author[1]{Kristian Uldall Kristiansen\footnote{Corresponding author, krkri@dtu.dk}}
\author[2]{Morten Gram Pedersen}
\affil[1]{Department of Applied Mathematics and Computer Science, 
Technical University of Denmark, 
2800 Kgs. Lyngby, 
Denmark}
\affil[2]{Department of Information Engineering,
University of Padova,
35131 Padova, 
Italy}

\begin{document}
\maketitle

\begin{abstract}
In this paper, we use geometric singular perturbation theory and blowup, as our main technical tool, to study the mixed-mode oscillations (MMOs) that occur in two coupled FitzHugh-Nagumo units with \response{symmetric and} repulsive coupling. In particular, we demonstrate that the MMOs in this model are not due \response{to} generic folded singularities, but rather due to singularities at a cusp -- not a fold -- of the critical manifold. Using blowup, we determine the number of SAOs analytically, showing -- as for the folded nodes -- that they are determined by the Weber equation and the ratio of eigenvalues. We also show that the model undergoes a (\response{symmetric}) saddle-node bifurcation in the desingularized reduced problem, which -- although resembling a folded saddle-node (type II) at this level -- also occurs on a cusp, and not a fold. We demonstrate that this bifurcation is associated with the emergence of an invariant cylinder, the onset of SAOs, as well as SAOs of increasing amplitude. We relate our findings with numerical computations and find excellent agreement.  
\end{abstract}



\newpage

\section{Introduction}
Coupled nonlinear oscillators are ubiquitous in physics, chemistry, biology and many other contexts. 
Interestingly, the collective behavior of the population \response{of oscillators} may exhibit qualitatively different dynamics that the individual units would if uncoupled. 
Coupling may, e.g., lead to oscillator death \cite{bar85,ermentrout90} or, on the contrary, promote oscillatory activity \cite{gregor10,wang92,weber12}.
In neurons and other cells capable of exhibiting complex {bursting} electrical activity, gap junction coupling can change the cellular behavior from a simple action potential firing to bursting \cite{sherman92,sherman94,devries00,pedersen05b,loppini15} and lead to large increases in the burst period  \cite{loppini18}. 


A particular kind of complex dynamics, also observed in models of cellular electrical activity, consist of mixed-mode oscillations (MMOs) where small- and large-amplitude oscillations (SAOs and LAOs, respectively) alternate \cite{brons06,desroches12}. 
Such dynamics is caused by cellular mechanisms operating on different time scales and can be seen, e.g., in the classical Hodgkin-Huxley model \cite{hodgkin52} for neuronal action potential generation \cite{rubin07}, and in experimental data and models of cortical neurons \cite{gutfreund95}, stellate cells \cite{dickson00,rotstein06},  
neuroendocrine cells \cite{tabak07,vo10,riz14,battaglin21}, cardiac cells \cite{yaru21,kimrey20}, among others. 
The mathematical structure causing MMOs is increasingly well understood by Geometric Singular Perturbation Theory (GSPT henceforth) \cite{fen3,jones_1995}  and often involves folded singularities, canard orbits \cite{szmolyan_canards_2001,wechselberger_existence_2005} and singular Hopf bifurcations \cite{guckenheimer08}, which are generally related to saddle-node bifurcations in the fast subsystem when treating slow variables as parameters \cite{brons06,desroches12}. 

In our recent study of coupled bursting oscillators \cite{pedersen22}, we revisited the finding by  Sherman \cite{sherman94} who showed that coupling of spiking cells can lead to bursting via slow desynchronization so that each burst is preceded by a large number of full action potentials (spikes). Before the transition to bursting, the averaged membrane potentials show SAOs reflecting amplitude-modulated spiking \cite{pedersen22}. 
We showed that the dynamical structure of  the system obtained by averaging was captured by a system of two coupled FitzHugh-Nagumo (FHN) units \cite{fitzhugh61,nagumo62}:
\begin{equation}\eqlab{fhn}
\begin{aligned}
 \dot v_1&=-v_1^3 +3v_1-w_1+g(v_2-v_1),\\
 \dot v_2&=-v_2^3+3v_2-w_2+g(v_1-v_2),\\
 \dot w_1 &=\epsilon(v_1-c),\\
 \dot w_2 &=\epsilon(v_2-c),
\end{aligned}
\end{equation}
with \response{symmetric and} repulsive coupling $g<0$, and that this simple system exhibits MMOs organized by a singular Hopf bifurcation related to a folded \response{singularity} \cite{guckenheimer08}. However, this \response{bifurcation was not related to a transcritical bifurcation of a folded node (the folded saddle-node \cite{krupa2010a}), as is typically seen in applications \cite{desroches12}, but rather to a cusp catastrophe in the fast subsystem. This observation motivated the current study of what we will refer to as cusped singularities.} 


\subsection{Background}\label{sec:2}
\response{In this paper, we continue our study of the two identical FHN units \eqref{fhn}. 
We start by highlighting three separate properties. Firstly, for $g=0$ then $(v_1,w_1)$ and $(v_2,w_2)$ decouple as a Lienard equation:
\begin{equation}\eqlab{fhnuncp}
\begin{aligned}
 \dot v_i &= -v_i^3 +3v_i-w_i,\\
 \dot w_i &=\epsilon (v_i-c),
\end{aligned}
\end{equation}
for $i=1,2$, 
and the dynamics of each pair is identical, being oscillatory (through relaxation oscillations) for $c\in (-1,1)$ and nonoscillatory (through a globally attracting equilibrium) for $c>1$ (and $c<-1$) for all $0<\epsilon\ll 1$ \cite{fitzhugh61,izhikevich07}, see \figref{uncoupled}. }

\begin{figure}[!h]
        \centering
        \includegraphics[width=.65\textwidth]{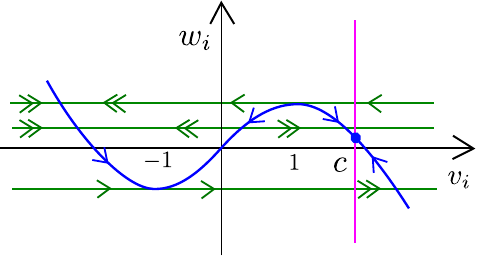}
        \caption{%
        \response{%
        Slow-fast dynamics of \eqref{fhn} in the uncoupled case $g=0$. The blue curve is the (cubic) critical manifold, whereas the vertical purple line is the $w_i$-nullcline. The case illustrated is for $c>1$ where there is an attracting equilibrium of the reduced problem on the rightmost stable branch of the critical manifold. For $c=1$ there is a canard point at the right most fold, leading to stable relaxation oscillations for $c\in (-1,1)$ and all $0<\epsilon\ll 1$.
        }
        }
        \figlab{uncoupled}
%
\end{figure}

\response{Secondly, the system \eqref{fhn} is symmetric with respect to $$\mathcal S:\,(v_1,v_2,w_1,w_2)\mapsto (v_2,v_1,w_2,w_1),$$ so that  $v_1=v_2$, $w_1=w_2$, being the fixed point of this symmetry, defines an invariant subspace. This will play an important role in the following. }

Finally, there is a unique equilibrium  
\begin{align}\eqlab{eqtrue}
q:\quad (v_1,v_2,w_1,w_2)=
(c,c, -c^3+3c, -c^3+3c)
,
\end{align}
of \eqref{fhn} and this point $q$ lies on the symmetric subspace defined by $v_1=v_2$, $w_1=w_2$. 
The Jacobian evaluated at $q$
has eigenvalues
\begin{eqnarray}
\nu_{1,2} &=& \frac{3-3c^2 \pm \sqrt{(3-3c^2)^2 -4\epsilon} }{2}\ ,\\ 
\nu_{3,4} &=& \frac{3-3c^2-2g \pm \sqrt{(3-3c^2-2g)^2 -4\epsilon} }{2}.
\end{eqnarray}
Let 
\begin{align}
v_s(g):=\sqrt{1-\frac23g}.\eqlab{vs}
\end{align}
Then a Hopf bifurcation occurs for $c=v_s(g)$ where $\nu_{3,4}$ are purely imaginary ($ \pm i\sqrt{\epsilon}$).
In fact, a direct calculation (\response{based upon center manifold and normal form theory, see specifically \cite[Equation 3.4.11]{guckenheimer97} and \remref{lyapunov} below}) shows that the associated first Liapunov number is given by
\begin{align}\eqlab{lyapunov}
l_1(\epsilon) = \frac{3(g-3)}{8g\sqrt{\epsilon}}(1+\mathcal O(\epsilon)).
\end{align}
Seeing that $l_1(\epsilon)>0$ for $g<0$ and all $0<\epsilon\ll 1$, it follows that we have a (singular) subcritical Hopf bifurcation \cite{guckenheimer97,guckenheimer08} for all $\epsilon>0$ small enough. 

In \cite{pedersen22}, it was observed numerically that the system \eqref{fhn} for \response{$c<v_s$ but $c\approx v_s$} and $0<\epsilon\ll 1$ exhibits mixed-mode oscillations (MMOs) with \response{an increasing number of} small-amplitude oscillations (SAOs) as $c$ approaches $v_s$ from below, 
see \figref{MMO}, \response{\figref{MMOzoom},} \figref{MMO_v1v2} and the figure captions. 
\response{Following large-amplitude oscillations (LAOs), the two cells almost synchronize ($v_1\approx v_2$). However, as the voltages approach $c$, they begin to diverge as they spiral apart, creating SAOs with increasing amplitudes before departing into additional large-amplitude excursions.
Interestingly, for some $c$ values, e.g., $c=1.27$, there is an alternation between $v_1$ and $v_2$ being increasing (and $v_2$, respectively, $v_1$ being decreasing) at the beginning of the LAOs, whereas this does not occur for other $c$ values. We will show that this phenomenon corresponds to the system leaving the neighborhood of the cusped singularity in different directions, a behavior which is not possible for the standard folded node. We return to this point in the Conclusions.
}

\begin{figure}[!h]
        \centering
	\includegraphics[width=\textwidth]{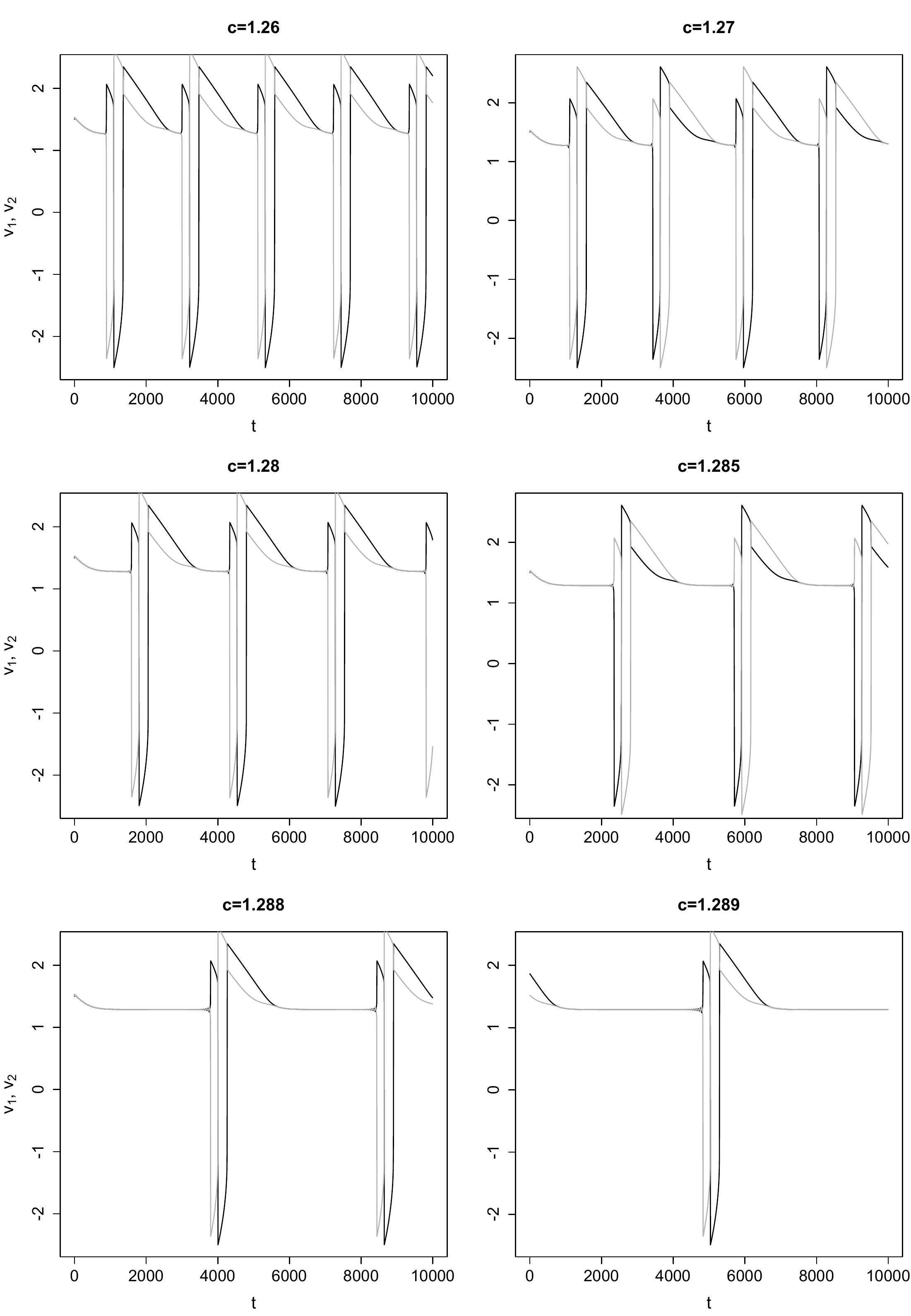}
        \caption{
        MMOs for $g=-1$, $\epsilon=0.01$ and $c$ values as indicated. 
        Black and gray curves show, respectively, $v_1(t)$ and $v_2(t)$.
        \response{For a zoom on the SAOs, see \figref{MMOzoom}.}
        }
        \figlab{MMO}
%
\end{figure}

\begin{figure}[!h]
        \centering
	\includegraphics[width=\textwidth]{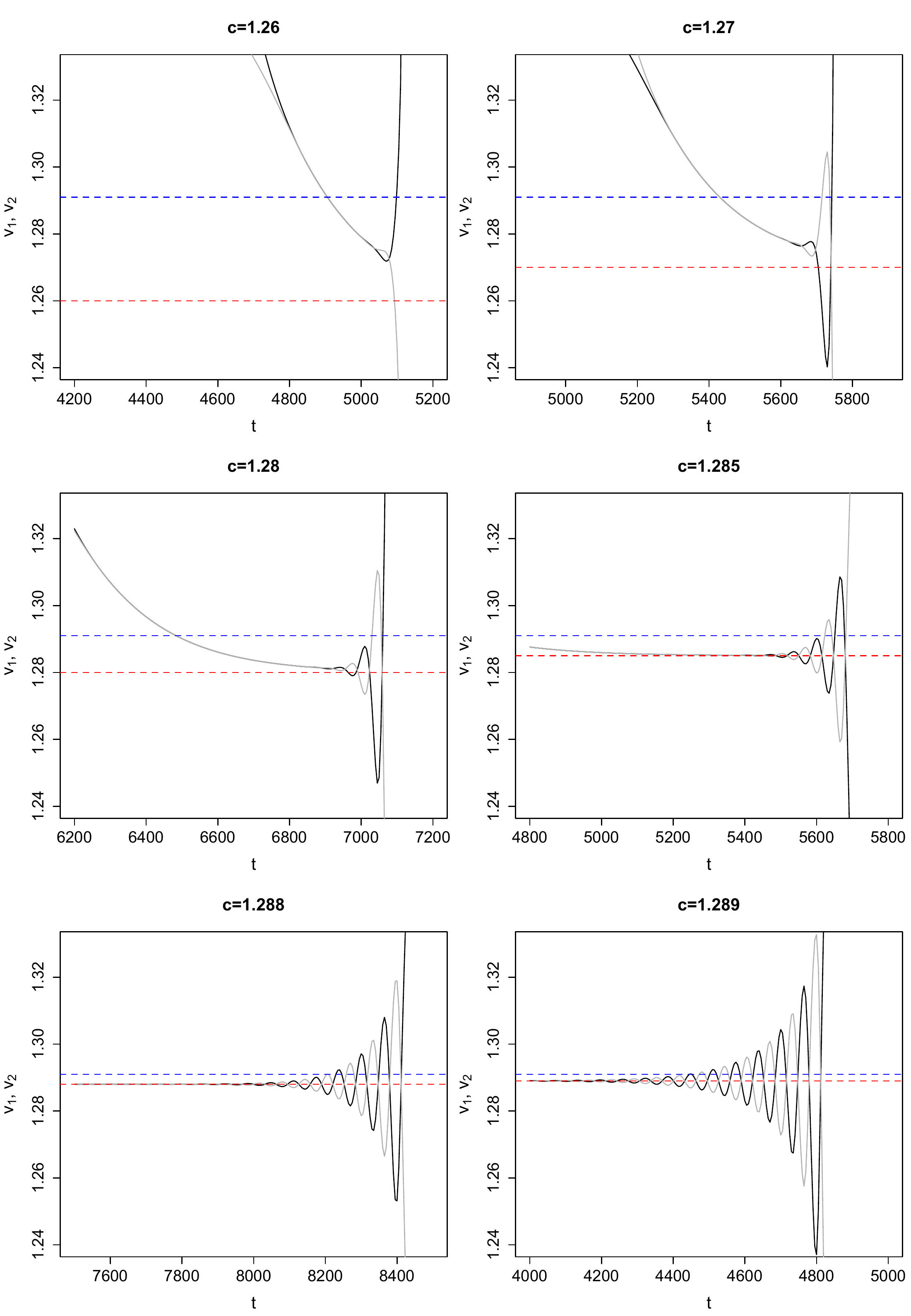}
        \caption{\response{Zoom on the SAOs in \figref{MMO} near $v_s=\sqrt{5/3}\approx 1.291$ (blue dashed line)
        with $c$ values (red dashed) as indicated. 
        Note how the number of SAOs increases as $c$ approaches $v_s$.}
        }
        \figlab{MMOzoom}
%
\end{figure}

\begin{figure}[!h]
        \centering
        \includegraphics[width=\textwidth]{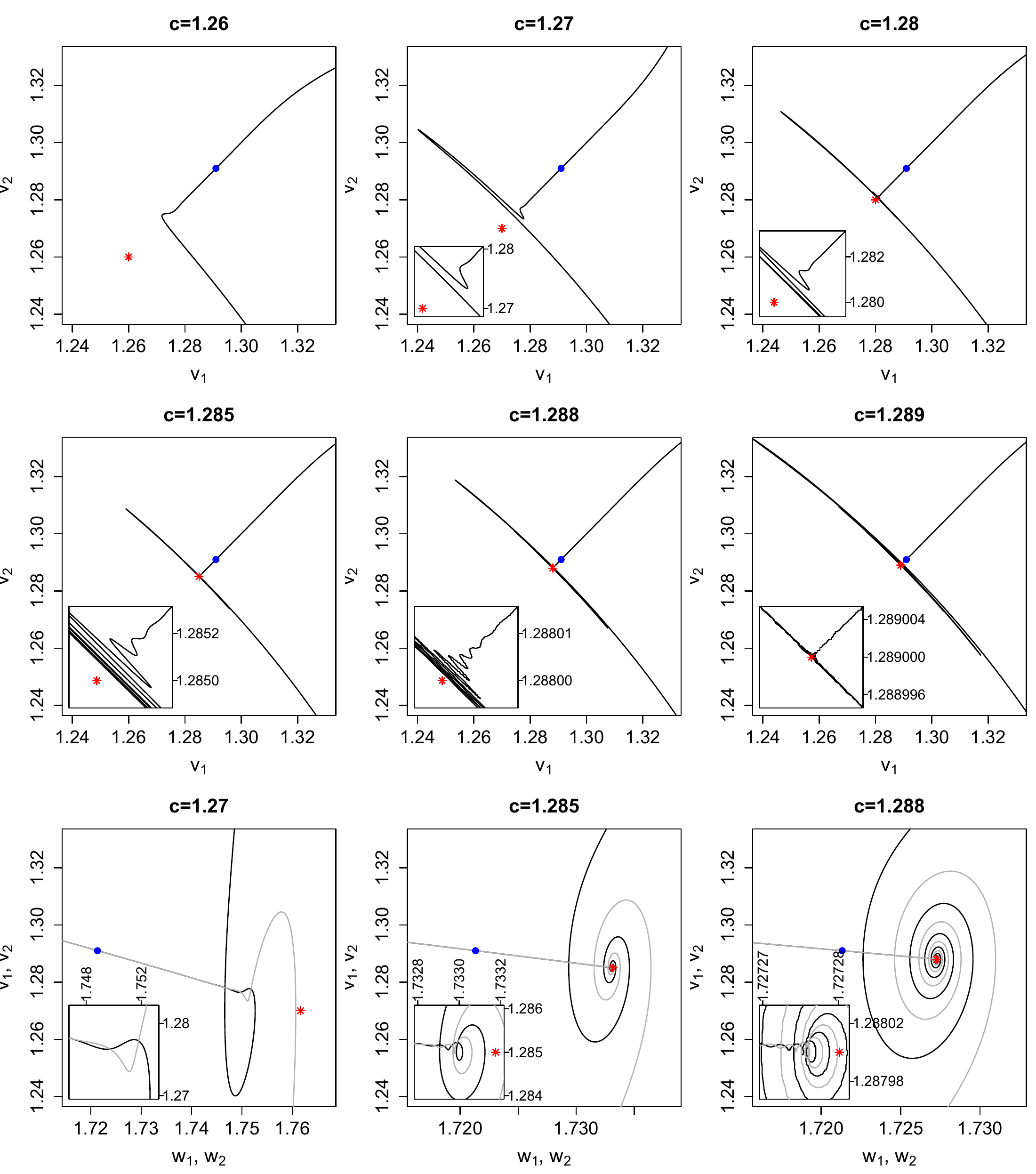} 
        \caption{SAOs near the folded singularity $(v_1,v_2,w_1,w_2)=(v_s,v_s,w_s,w_s)$ where $w_s=-v_s^3+3v_s$ (blue dot; see \propref{foldedsing}). The full-system saddle point \eqref{eqtrue} is shown as a red asterisk. 
        In the lower panels, the black and gray curves show, respectively, $(w_1,v_1)$ and $(w_2,v_2)$.
        \response{Insets show zooms on the SAOs near the saddle point.}
        Parameters as in \figref{MMO}.
        }
        \figlab{MMO_v1v2}
        
\end{figure}
\begin{figure}[!t]
        \centering
        \includegraphics[width=0.95\textwidth]{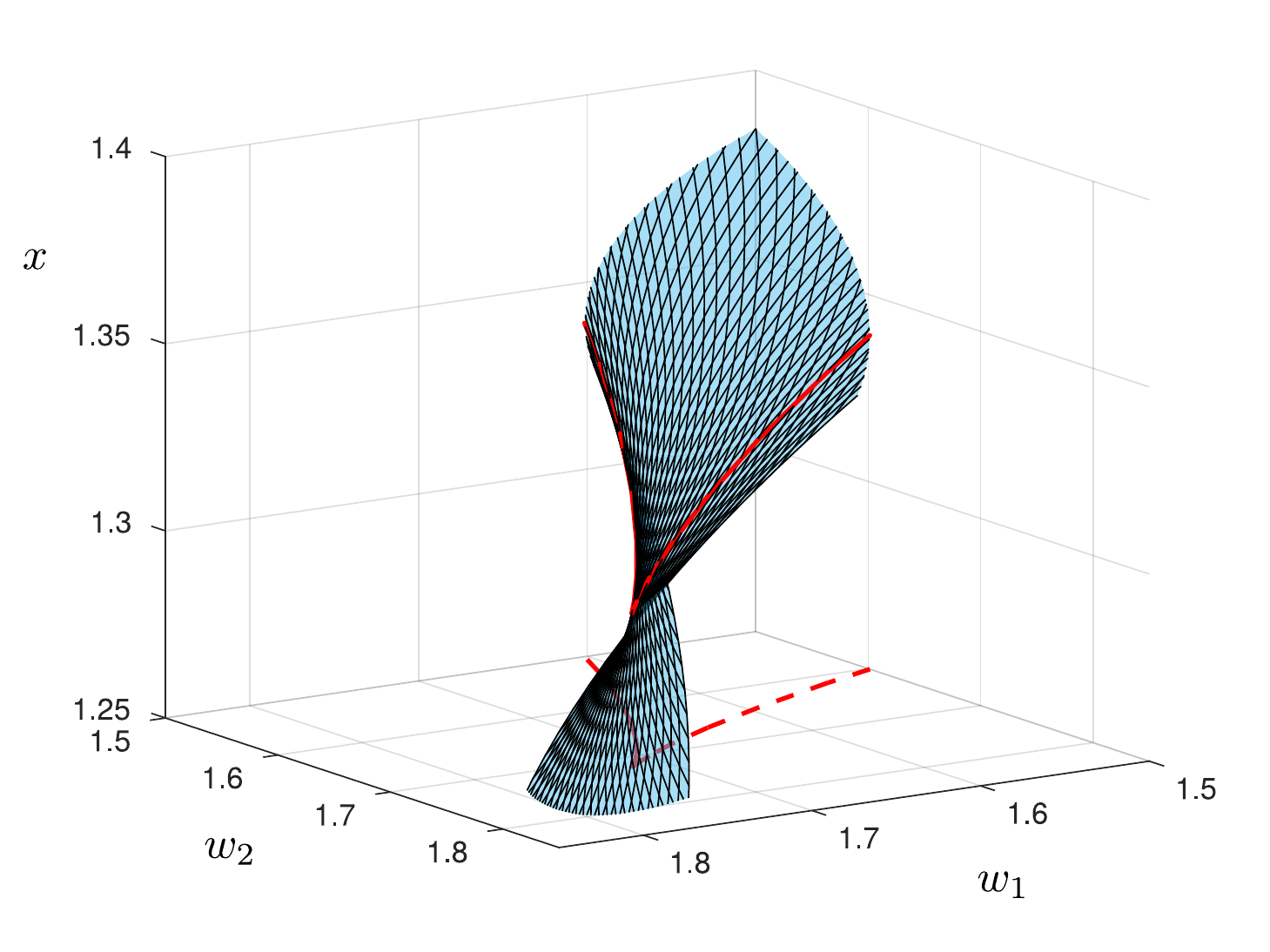}
        \caption{\response{The two-dimensional critical manifold $C$ of \eqref{fhn} viewed in the $(w_1,w_2,x)$-space with $x=\frac12 (v_1+v_2)$ for $g=-1$. The red curve is a set of non-normally hyperbolic points. The projection onto the $(w_1,w_2)$-plane shows the cusp singularity. \eqref{fhn} has a cusped node (cusped saddle-node) for $c\approx v_s$ but $c<v_s$ ($c=v_s$, respectively), insofar that the reduced problem has a folded singularity on the cusp of nodal type (saddle-node type, respectively) upon desingularization.} }
        \figlab{cusp_new}
\end{figure}

\subsection{Biophysical motivation and implications}\seclab{motivation}
\response{Since the FHN model is a simplification of the Hodgkin-Huxley model for neuronal activity, our results have implication for neuroscience beyond providing insight into our previous study \cite{pedersen22}.}
\response{Negative coupling ($g<0$) resembles mutual inhibition, for example between neuronal populations, see e.g. Curtu and Rubin \cite{curtu10,curtu11}. These authors showed that MMOs can appear as a result of inhibition via a singular Hopf bifurcation.
Mutual inhibition has also been used to explain ``binocular rivalry", where 
 perception alternates between different images presented to each eye
\cite{laing02}.

As explained above, our choice of $c>1$ (similarly, one could consider $c<-1$) means that the FHN neurons are silent when uncoupled. 
Our results show that repulsive coupling can induce oscillatory activity in such otherwise silent neurons via MMOs related to cusped singularities.
These results mimic previous findings for ``release'' and ``escape" mechanisms generating oscillations in a couple of inhibitory non-oscillatory neurons \cite{wang92}.
We do not consider $-1<c<1$ or $g>0$ since the system in these cases does not present a cusped singularity producing SAOs, which is the main topic of the manuscript.
}

 \subsection{Main results}\seclab{mainres}
 In this paper, we describe the origin and mechanisms underlying the MMOs in \figref{MMO}, \response{\figref{MMOzoom} and \figref{MMO_v1v2}}. In particular, we show that the coupled system \eqref{fhn} possesses a degenerate folded singularity in the singular limit $\epsilon\rightarrow 0$, and demonstrate -- through a center manifold computation -- that this singularity corresponds to a cusp, see also \figref{cusp_new} and the figure caption for details. Moreover,
by performing a detailed blow-up analysis, we dissect the details of the dynamics near this new type of singularity and show that it lies at the heart of the mechanism causing SAOs. As for folded singularities \cite{szmolyan_canards_2001}, we divide our analysis into two parts: one part covering the generic case (i.e. without an additional unfolding parameter) and one covering the bifurcation in the presence of an unfolding parameter. Since, the former case resembles the folded node \cite{wechselberger_existence_2005} -- in particular, we will show that the number of SAOs is also determined by the Weber equation and the ratio of eigenvalues -- we will refer to this singularity as the ``cusped node singularity''. Similarly, our results also show that the degenerate case, which we name the ``cusped saddle-node singularity'', as the classical folded saddle-node \response{(of type II \cite{krupa2010a})}, marks the onset of SAOs in the coupled FHN system.

\response{\response{For any $b\ge 0$, let $\lfloor b \rfloor,$ denote the largest integer $n\in \mathbb N_0$ such that $n\le b$.} We then summarize our findings on the SAOs in the following theorem (we refer to \thmref{main1} and \thmref{main2} for more detailed versions).
\begin{theorem}\thmlab{thm0}
Consider \eqref{fhn} with $g<0$ and suppose \begin{align}
  c\in \left(\frac{1-\frac12 g}{\sqrt{1-\frac{2g}{3}}},v_s\right).\eqlab{cinterval0}
  \end{align}
 Then we have (\textnormal{the cusped-node case}):
 \begin{enumerate}[series=part1]
   \item There is a desingularization of the reduced problem on the critical manifold $C$ of the slow-fast system \eqref{fhn} such that the system has an attracting singularity $f_1$, given by 
   \begin{align*}
    v_i  = v_s, \quad w_i = w_s,\quad \text{for}\quad i=1,2,
   \end{align*}
   for $w_s:=-v_s^3+3 v_s$,
   that lies on a cusp of $C$. Moreover, the linearization around $f_1$ has the following eigenvalues
    \begin{align}
 \lambda_1: = -6v_s(v_s-c),\quad \lambda_2: = -\lambda_1 +2g. \eqlab{eigvals} 
\end{align}
with $\lambda_2<\lambda_1<0$ for the values in \eqref{cinterval0}.
      The point $f_1$ is therefore a stable node for the desingularized system; specifically, it (locally) attracts all points on the attracting subset of $C$. 
 \item \label{item2} Orbits of \eqref{fhn} that pass through $f_1$ for $0<\epsilon\ll 1$ will (in general) undergo SAOs around the symmetric subspace $v_1=v_2$, $w_1=w_2$ before leaving a neighborhood of $f_1$. 
 \item Suppose that $ \frac{\lambda_2}{\lambda_1}\notin \mathbb N$. Then the amplitude of these SAOs is of the order $\mathcal O(\epsilon^{\frac{\lambda_2}{2\lambda_1}})$ and the number of SAOs is given by $\lfloor \frac{\lambda_2}{\lambda_1}\rfloor$ many full $180^\circ$ rotations around the symmetric subspace $v_1=v_2$, $w_1=w_2$, for all $0<\epsilon\ll 1$.
 \end{enumerate}
 Next, fix any 
 \begin{align}
  c_2 \in \left(-\frac{1}{3v_s},0\right),\eqlab{c2interval0}
 \end{align}
and consider 
\begin{align}
 c = v_s+\sqrt{\epsilon} c_2.\eqlab{c2scaling0}
\end{align}
Then we have (\textnormal{the cusped-saddle node case}):
 \begin{enumerate}[resume*]
 \item The singularity $f_1$ of the desingularized reduced problem on $C$ is a saddle-node for $c=v_s$ with $\lambda_2<\lambda_1=0$, locally attracting all points on the attracting subset of $C$. 
 \item \label{item5} With $c$ as in \eqref{c2scaling0} and $c_2$ fixed in \eqref{c2interval0}, the regular singularity $q$, given by \eqref{eqtrue}, is of saddle-focus type for all $0<\epsilon\ll 1$, having a two-dimensional unstable manifold with focus-type dynamics.
 \item Orbits of \eqref{fhn}, with $c$ as in \eqref{c2scaling0} and $c_2$ fixed in \eqref{c2interval0}, that pass through $f_1$ for $0<\epsilon\ll 1$ will (in general) undergo SAOs around the symmetric subspace $v_1=v_2$, $w_1=w_2$, before leaving a neighborhood of $f_1$. 
 \item There are finitely many of the SAOs that are $\mathcal O(\epsilon^{1/4})$ in amplitude as $\epsilon\rightarrow 0$. 
 \item The number of SAOs with an amplitude that is exponentially small with respect to $\epsilon>0$ is unbounded  as $\epsilon\rightarrow 0$.
 \end{enumerate}
 There are no SAOs for $c_2>0$ and all $0<\epsilon\ll 1$.
\end{theorem}
}
%

Despite the similarities between the folded and cusped versions of the singularities, we will also illuminate some differences. For example, we show that the cusped saddle-node is intrinsically related to a regular Lienard equation in the same way that the FSN is related to the canard explosion. See \lemmaref{lienard} and \propref{contractg2} for details.

\begin{remark}\label{remark1}
The reference \cite{krupa14} also considers coupled oscillators in four dimensions, including systems like \eqref{fhn}.  The focus is (also) on emergence of MMOs in these types of systems. However, in contrast to our work, the singularities in \cite{krupa14} are folded and not ``cusped'' and the analysis of the coupled FitzHugh-Nagumo system, see \cite[Section 4]{krupa14}, also focuses on the attractive coupling $g>0$. In the present paper, we will only consider $g<0$.
\end{remark}

\subsection{Numerical results}\seclab{numerics}
\response{To illustrate \thmref{thm0}, we compare the theoretical results to numerical simulations. 
We define
\begin{equation}\eqlab{yuz}
u=\frac12(v_1-v_2),\quad y=\frac12(w_1+w_2)-w_s,  \quad z=\frac12(w_1-w_2).
\end{equation}
%
In \thmref{thm0}, we count the number of SAOs as $180^\circ$-rotations around the symmetric subspace. In the coordinates \eqref{yuz}, the symmetric space corresponds to $u=0$, $z=0$, so when projected onto the $(u,z)$-plane, SAOs correspond to full $180^\circ$ rotations around the origin.
Therefore we have the following: \textit{The number of SAOs is one less than the number of zeros of $u=0$}. We illustrate this in \figref{c124} for $c=1.24$, $g=-1$ and $\epsilon=0.01$. Here $\frac{\lambda_2}{\lambda_1}=4.06$ and we find five simple zeros of $u$ and in agreement with \thmref{thm0} precisely four full $180^\circ$-rotations. } 

 \response{On the other hand, \figref{twists_ampl}A shows a typical orbit of the system \eqref{fhn} in the $(u,y,z)$-space for $c=1.28$, $g=-1$, $\epsilon=0.01$. It corresponds to the cusped saddle-node case. 
The orbit approaches \response{the symmetric subspace $u=0$, $z=0$ denoted by $\gamma$} (red dotted line), 
and moves towards and beyond the cusped singularity $f_1$ located at the origin (blue dot), coming close to the saddle-focus point $q$ (red asterisk) before spiralling outwards.
%
To find the number of SAOs as a function of $c$,  we counted the number of zeros of $u=0$ as asymptotes of $z/u$ (see \figref{twists_ampl}B) for a range of $c$ values.
These numerical results were then plotted against the theoretical values of \response{item 3~in \thmref{thm0}} 
 using the explicit expressions for the eigenvalues \eqref{eigvals}, see \figref{twists_ampl}C.}\response{ The correspondence is excellent with minor discrepancies for $c$ values in the interval given by \eqref{c2interval0}-\eqref{c2scaling0}, the cusped-saddle node region, where \thmref{thm0} predicts an unbounded number of exponentially small SAOs as $\epsilon\rightarrow 0$ (see \thmref{thm0}, item 8). This is not a surprise, as \thmref{thm0}, item 3, assumes that $c$ is uniformly bounded away from $v_s$. The increment in amplitude of the SAOs as $c$ increases \response{and enters the cusped-saddle node region}, due to focus dynamics near $q$, see \thmref{thm0}, items 5 and 6, is also confirmed by the simulations (\figref{twists_ampl}D).}

   \begin{figure}[!t]
        \centering
        \includegraphics[width=.98\textwidth]{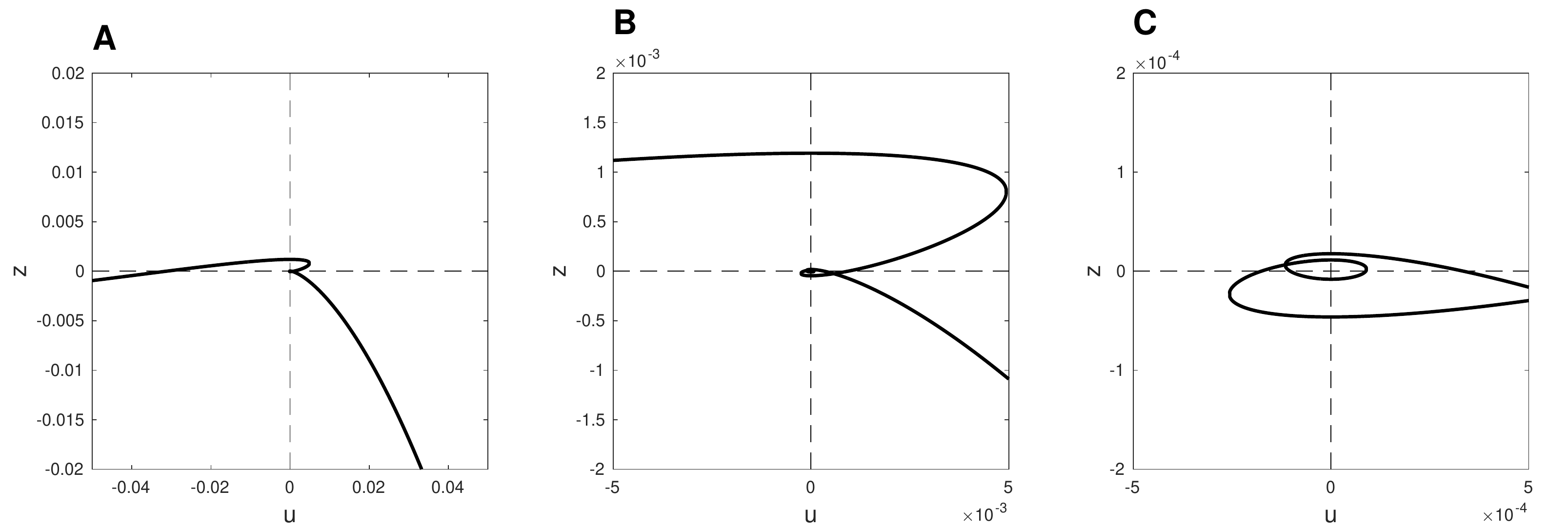}
        \caption{%
        \response{Simulated trajectory projected onto the $(u,z)$-plane for $c=1.24$ and $g=-1$ (the cusped node case). Here $\frac{\lambda_2}{\lambda_1}=4.06$. Panels B and C show zooms on the SAOs near the origin. In agreement with \thmref{thm0}, there are four full $180^\circ$-rotations (twists), corresponding to five simple zeros of $u$, for this value of $c$. Only three of the twists are visible in the zoom C, the fourth, larger one can be seen in panel B.}
        }
        \figlab{c124}
\end{figure}

   \begin{figure}[!t]
        \centering
        \includegraphics[width=.87\textwidth]{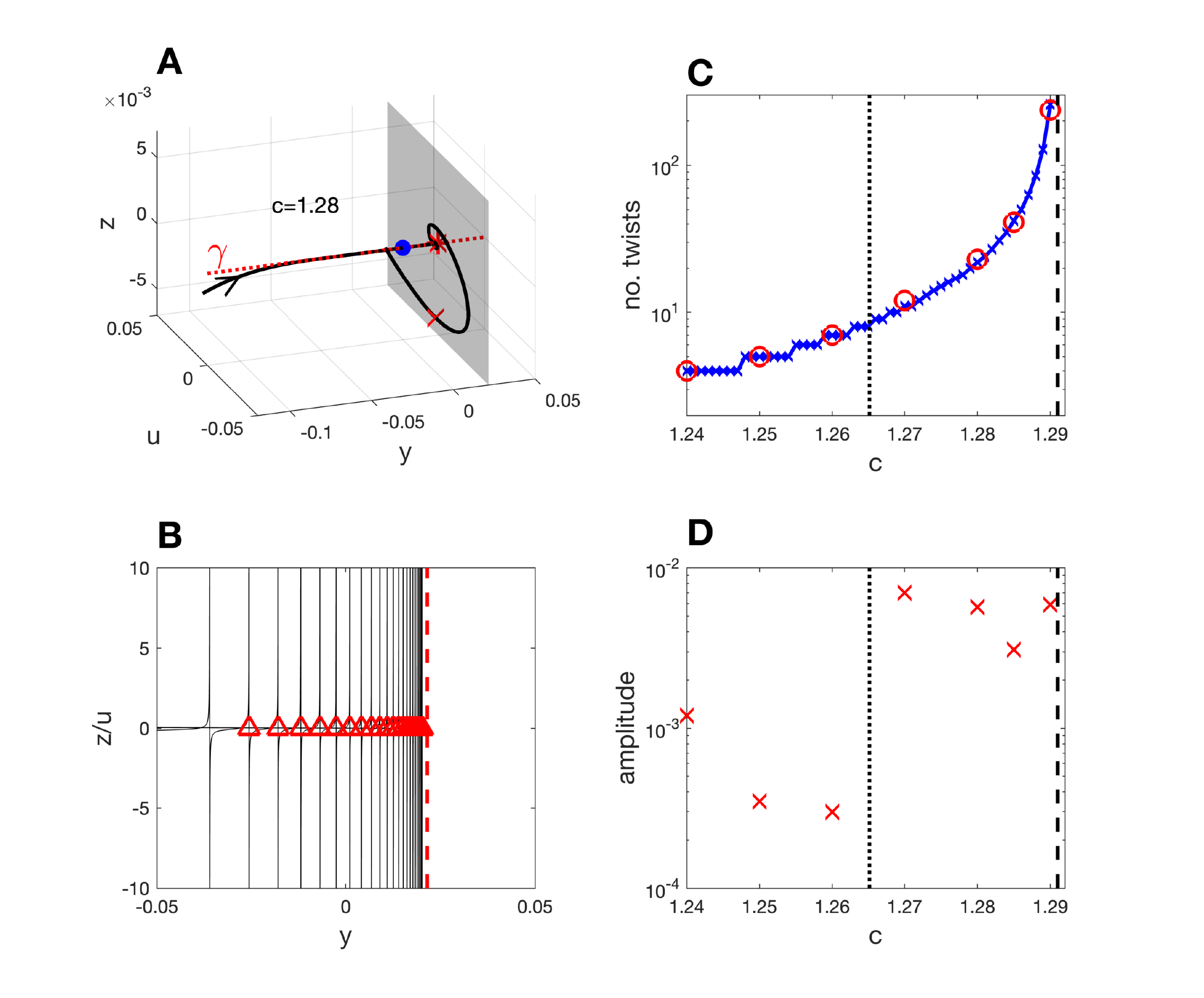}
        \caption{%
        \textbf{A}: \response{Simulated trajectory (black curve) in $(y,u,z)$-space, see \eqref{yuz},} with parameters as in \figref{MMO} and $c=1.28$. 
        The 
        blue point is the cusped singularity $f_1$. The red asterisk is the full-system saddle-focus point $q$, which lies in the gray plane given by $y=-c^3+3c-w_s$. The red cross indicates the point used to find the amplitude of the $z$-value at the last twist, cf. panel D. The red dotted line is $\gamma$ ($u=z=0$).
        \textbf{B}: Simulation as in panel A but showing $z/u$ as function of $y$ in order to count the number of twists around $\gamma$, see \response{\thmref{thm0}, item 2.} 
        The red triangles show the asymptotes (where $u=0$), which were found automatically. The number of SAOs is one less than the number of asymptotes (see the text for further details), which is why the first asymptote is not indicated. 
        The vertical red dashed line indicates the $y$-value of the saddle-focus. 
        \textbf{C}: Number of SAOs predicted from \response{\thmref{thm0}, item 3} 
        (blue crosses), for a series of $c$ values compared to the number of SAOs found  from the simulations (red circles), as indicated by the triangles in panel B. 
        The vertical dotted line at $c=v_s-\frac{\sqrt{\epsilon}}{3v_s}\approx 1.265$ indicates the \response{left boundary of the cusped-saddle node case,
        see \thmref{thm0},} 
        whereas the dashed line is $c=v_s$.
        \textbf{D}: The amplitude of the last rotation (red crosses), estimated as the absolute value of $z$ at the last relevant asymptote, i.e., the right-most triangle in panel B, see the red cross in panel A. 
        The vertical lines are as in panel C. Note how the amplitude increases dramatically as $c$ approaches $v_s$ and enters the \response{saddle-node} region with  \response{$\mathcal O(\epsilon^{1/4})$-amplitude} SAOs.
        }
        \figlab{twists_ampl}
\end{figure}

\subsection{Overview}\seclab{overview}
In Section \ref{sec:3}, we first study \eqref{fhn} as a singular perturbation problem for $\epsilon\rightarrow 0$ using GSPT. Specifically, we present a complete analysis of the reduced problem for any \response{fixed $g<0$} and describe all bifurcations for $c>0$ at the singular level. This then leads to a local three-dimensional center manifold reduction (\response{with parameters $\epsilon,c$ and $g$}) in \propref{cmred}. \response{In \lemmaref{cuspsing}, we then show that the critical manifold of this reduced system has a cusp singularity}. In Sections \ref{sec:4} and \ref{sec:5}, we proceed to study the dynamics near the cusped node and the cusped saddle-node singularity, respectively, by using the blowup method \cite{dumortier1996a,krupa_extending_2001} as the main technical tool. 
\response{This leads to \thmref{main1} and \thmref{main2} describing the SAOs in the two scenarios. The two theorems imply \thmref{thm0}.} In Section \ref{sec:final}, we conclude the paper. 

\section{GSPT-analysis of \eqref{fhn}}\label{sec:3}
To analyze \eqref{fhn} as a slow-fast system, we first study the layer problem and the reduced problem. The layer problem is obtained by setting $\epsilon=0$ in \eqref{fhn}:
\begin{equation}\eqlab{layer}
\begin{aligned}
 \dot v_1&=-v_1^3 +3v_1-w_1+g(v_2-v_1),\\
 \dot v_2&=-v_2^3+3v_2-w_2+g(v_1-v_2),\\
 \dot w_i &=0,
 \end{aligned}
 \end{equation}
for $i=1,2$. On the other hand, the reduced problem, \response{given by}:
\begin{equation}\eqlab{reduced}
\begin{aligned}
 0 &= -v_1^3+3v_1-w_1+g(v_2-v_1),\\
 0 &=-v_2^3 +3v_2-w_2+g(v_1-v_2),\\
 w_1'&=v_1-c,\\
 w_2'&=v_2-c,
\end{aligned}
\end{equation}
is obtained by setting $\epsilon=0$ in the slow time ($\tau = \epsilon t$) version of \eqref{fhn}:
\begin{equation}\nonumber
\begin{aligned}
 \epsilon v_1'&=-v_1^3 +3v_1-w_1+g(v_2-v_1),\\
 \epsilon v_2'&=-v_2^3+3v_2-w_2+g(v_1-v_2),\\
 w_1' &=v_1-c,\\
 w_2' &=v_2-c,
\end{aligned}
\end{equation}
where $()'=d/d\tau$. In the following, we will analyze \eqref{layer} and \eqref{reduced} successively. 

\subsection{Analysis of the layer problem \eqref{layer}}
The equilibria \response{of} the layer problem are given by
\begin{align*}
  0 &= -v_1^3+3v_1-w_1+g(v_2-v_1),\\
 0 &=-v_2^3 +3v_2-w_2+g(v_1-v_2).
\end{align*}
This defines a two-dimensional critical manifold  \response{$C$} of \eqref{layer} in the four-dimensional phase space. The manifold $C$ can be written as a graph $w=h(v)$ over $v$ where $h=(h_1,h_2)$ with
\begin{align*}
 h_1(v_1,v_2):=-v_1^3+3v_1+g(v_2-v_1),\\
 h_2(v_1,v_2):=-v_2^3+3v_2+g(v_1-v_2).
\end{align*}
We determine the stability of $C$ by linearizing the layer problem \eqref{layer} around any point $(v,h(v))\in C$. It is a basic fact, that the nontrivial eigenvalues are given by the eigenvalues of the Jacobian
\begin{align*}
 Dh(v_1,v_2) = \begin{pmatrix}
                -3v_1^2-g+3 & g\\
                g &-3v_2^2-g+3
               \end{pmatrix}.
\end{align*}
\response{The matrix is symmetric, so the eigenvalues are real. Moreover, we have}
\begin{align*}
 \text{tr}\,Dh &= -3(v_1^2+v_2^2)+6-2g,\\
 \text{det}\,Dh &=9 v_1^2v_2^2 \response{-}3(3-g) (v_1^2+v_2^2)+3(3-2g).
\end{align*}
%
\response{Consequently, $\text{tr}\,Dh=0$ defines a circle centered at $(0,0)$ with radius $\sqrt{2-\frac23 g}$. On the other hand, $\text{det}\,Dh=0$ can be written in the polar coordinates $(r,\theta)$: $v_1=r\cos \theta$, $v_2=r\sin \theta$ as 
\begin{align}
 \cos^2(\theta) \sin^2(\theta) r^4 -(1-\frac13 g)r^2 +1-2/3g=0,\eqlab{eqnr2}
\end{align}
which is a quadratic equation in $r^2$. }
\response{
\begin{lemma}\lemmalab{mums}
Consider \eqref{eqnr2} as an equation for $r>0$ and suppose that $g<0$. Then for each $\theta\ne n\pi/2$, $n\in \mathbb Z$,  there exists two solutions $r=m_u(\theta)$ and $r=m_s(\theta)$ with 
\begin{align}
0<m_u(\theta)<\sqrt{2-\frac23 g}<m_s(\theta),\eqlab{ineq}
\end{align}
where
 \begin{align*}
 m_u:&\, \mathbb R\rightarrow \mathbb R_+,\quad m_s:\,\mathbb R\backslash \{\theta\ne n\pi/2,n\in \mathbb Z\} \rightarrow \mathbb R_+,
 \end{align*}
 are smooth functions. For $\theta= n\pi/2$, $n \in \mathbb Z$, there is only one solution and it is given by $r=m_u(\theta)$. Finally, for each $n\in \mathbb Z$:
 \begin{align*}
 m_s(\pi/4+n\pi/2) = \min m_s = v_s,\quad 
  m_s(\theta)\rightarrow \infty\quad \text{for}\quad \theta\rightarrow n\pi/2.
 \end{align*}
\end{lemma}
\begin{proof}
 Follows from a direct calculation. In particular, $\theta= n\pi/2$, $n \in \mathbb Z$ are the values where the coefficient of $r^4$ vanishes. In order to obtain \eqref{ineq}, we use that
 the curves defined by $\text{tr}\,Dh=0$ (a circle with radius $r=\sqrt{2-\frac23 g}$) and $\text{det}\,Dh=0$ do not intersect. To see this, one can use that
 $$\text{tr}\,(Dh)^2-4\text{det}\,Dh=9(v_1^2-v_2^2)^2+4g^2>0.$$
 \end{proof}}
The expressions for $m_{u,s}$ are not important and therefore left out. 
\response{Following this lemma, we now define} $C_{RN}$, $C_{AN}$ and $C_S$ as the subsets of $C$ with $0\le r<m_u(\theta)$, $r>m_s(\theta)$ and $m_u(\theta)<r<m_s(\theta)$, respectively, in the polar coordinates $(r,\theta)$. 
Let also $F_i$ be the subset of $C$ defined by $r=m_i(\theta)$, $i=u,s$. Then 
\begin{align*}
 C = C_{RN} \cup F_u\cup C_S \cup F_s \cup C_{AN}.
\end{align*}
\begin{figure}[!h]
        \centering
        \includegraphics[width=.76\textwidth]{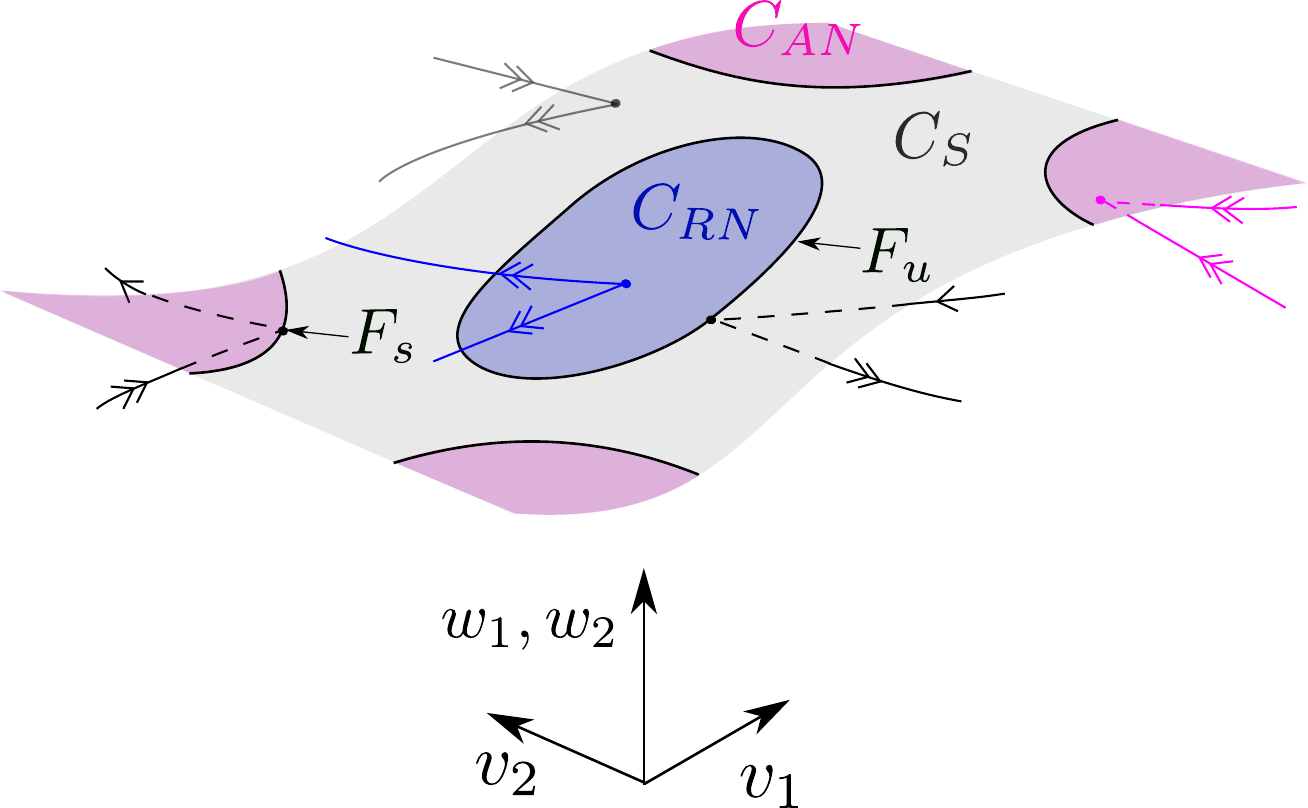}
        \caption{\response{Sketch of the critical manifold. The five sets of curves that are attached to different points on $C$ are orbits of the layer problem (contained within $w_i=\text{const}.$) and indicate the normal stability properties of $C$ along its different components: $C_{RN}$ (repelling, nodal-type), $C_S$ (saddle-type), $C_{AN}$ (attracting, nodal type), $F_s$ and $F_u$ (both saddle-nodes). Here we follow the standard convention that double-headed arrows indicate hyperbolic directions whereas single-headed arrows indicate center/slow directions.} }
        \figlab{lemma1}
%
\end{figure}
We then conclude the following \response{(see \figref{lemma1} for an illustration}):
\begin{lemma}\lemmalab{Fi}
$F_i$, $i=u,s$ are sets of loss of normal hyperbolicity, but (since $\textnormal{tr}\,Dh\gtrless 0$ on $F_i$) the linearization along $F_u$ has one positive eigenvalue whereas the nontrivial eigenvalue of the linearization along $F_s$ is negative. 
Moreover, we have the following classification. 
 \begin{itemize}
 \item  Suppose $(v,h(v))\in C_{RN}$. Then the eigenvalues of $Dh(v)$ are both positive and real and $v$ is therefore a \textnormal{repelling node} for the fast subsystem of \eqref{layer}.
\item Suppose $(v,h(v))\in C_{AN}$. Then the eigenvalues of $Dh(v)$ are both negative and real and $v$ is therefore an \textnormal{attracting node} for the fast subsystem of \eqref{layer}.
\item Suppose $(v,h(v))\in C_{S}$. Then the eigenvalues of $Dh(v)$ are real and have opposite signs and $v$ is therefore a \textnormal{saddle} for the fast subsystem of \eqref{layer}.
\end{itemize}
In particular, on $C_{RN}$ and $C_{AN}$, $\textnormal{det}\,Dh >0$ whereas $\textnormal{det}\,Dh<0$ on $C_S$. 
\end{lemma}
\subsection{Analysis of the reduced problem \eqref{reduced}}
The reduced problem \eqref{reduced} is defined on $C$. Since $C$ is a graph over $v$, we will write this system in terms of $v$ instead of $w$. This gives
\begin{align}
 Dh v' &=v-\mathbf{c},\eqlab{reducedC0}
\end{align}
with 
\begin{align}\eqlab{tbfx}
\textbf{c}:=(c,c),\end{align} a notation we adopt in the following. Using the \response{adjugate} matrix 
\begin{align*}
 \text{adj}\,Dh(v)=  \begin{pmatrix} 
                       -3v_2^2-g+3 & -g\\
                       -g & -3v_1^2-g+3
                      \end{pmatrix},
\end{align*}
of $Dh$, we may write this equation in the following equivalent form:
\begin{align}
 \text{det}\, Dh(v)\,v' &=\text{adj}\,Dh(v)\,(
v-\mathbf{c}).\eqlab{rp}
\end{align}
On $C_{RN}\cup C_{AN}$, where $\text{det}\,Dh>0$, recall \lemmaref{Fi}, we are therefore led to consider the equivalent system
\begin{align}\eqlab{desrp}
 \dot v &=\text{adj}\,Dh(v)\,(
v-\mathbf{c}).
\end{align}
Since $\text{det}\,Dh<0$ on $C_S$, the \textit{desingularized system} \eqref{desrp} is also equivalent to the reduced problem on $C_{S}$ upon time reversal.
Folded singularities, which organize SAOs and canard trajectories connecting attracting and repelling sheets of \response{the} critical manifold, are equilibria of \eqref{desrp} on $F_i$ where $\text{det}\,Dh =0$. \response{We then state and prove the following result, see also \figref{eq}}.
\begin{proposition}\proplab{foldedsing}
 Consider \eqref{rp} with $g<0$. \response{Then there is a regular singularity $q(c)$ at $v=\mathbf{c}$ for any $c>0$ and at most four folded singularities:
 \begin{enumerate}
  \item There are two folded singularities $f_1$ and $f_2$ that exist for all $c>0$,  occur on the symmetric subspace defined by $v_1=v_2$, and are given by $v=\pm  \mathbf{v}_s$ on $F_s$ where $$v_s(g)=\sqrt{1-\frac{2g}{3}}.$$
 \item \label{cle1} For $c<1$, then there are two separate folded singularities $f_3(c)$ and $f_4(c)$ that both lie on $F_u$, but outside the symmetric subspace (i.e $v_1\ne v_2$ along these), and are given by the equations
 \begin{equation}\eqlab{q3q4}
 \begin{aligned}
  \frac12 (v_1+v_2) &= \frac{gc}{3c^2+g-3},\\
  \frac14 (v_1-v_2)^2& = \left(\frac12 (v_1+v_2)-c\right)^2 +1-c^2.
 \end{aligned}
 \end{equation}
 \item \label{cge} For $c>\frac{1-\frac{g}{3}}{\sqrt{1-\frac{2g}{3}}}$, $c\ne \sqrt{1-\frac{g}{3}}$, then $f_3(c)$ and $f_4(c)$ (again given by the equations \eqref{q3q4}) are nonsymmetric folded singularities, now located on $F_s$. $f_3(c)$ and $f_4(c)$ go unbounded as $c\rightarrow \sqrt{1-\frac{g}{3}}$.
 \end{enumerate}
 The point $q(c)$ undergoes two pitchfork bifurcations of \eqref{desrp} at $c=1$ and $c=\frac{1-\frac{g}{3}}{\sqrt{1-\frac{2g}{3}}}$ (sub and super-critical, respectively, giving rise to $f_3(c)$ and $f_4(c)$ in items \ref{cle1} and \ref{cge}), and a transcritical bifurcation at $c=v_s$. }
 
%
%
%
\end{proposition}
\begin{proof}
 We use that folded singularities are equilibria of \eqref{desrp} where $v\ne \mathbf{c}$; $v=\mathbf{c}$ corresponds to the regular singularity $q(c)$. 
 
 Setting $v_1=v_2$, we then find the two (isolated) folded singularities \response{$f_1$ and $f_2$ given by} $v=\pm \mathbf{v}_s$ on $F_s$, recall \eqref{vs}. When $v_1\ne v_2$, we consider $u=\frac12 (v_1-v_2),x=\frac12 (v_1+v_2)$ with $u\ne 0$. This gives
 \begin{align}
  x=\frac{gc}{3c^2+g-3},\, u^2 =(x-c)^2+1-c^2.\eqlab{xu2fs}
 \end{align}
 for $c\ne \sqrt{1-\frac{g}{3}}$. \eqref{xu2fs} gives \eqref{q3q4} upon returning to $v_1,v_2$ and the existence of $f_3(c)$ and $f_4(c)$. 
 Setting $u=0$ in \eqref{xu2fs} gives $c=1,x=1$ and $c=\frac{1-\frac{g}{3}}{\sqrt{1-\frac{2g}{3}}},x=v_s$ for $c>0$. This gives the pitchfork bifurcations. \response{It is a direct calculation to verify the remaining claims regarding $f_3(c)$ and $f_4(c)$ of items \ref{cle1} and \ref{cge}.} 
\end{proof}
The linearization of \eqref{desrp} around $v=\mathbf{v}_s(g)$ produces the following eigenvalues:
\begin{align}
 \lambda_1: = -6v_s(v_s-c),\quad \lambda_2: = -\lambda_1 +2g.\eqlab{eigval}
\end{align}
It is possible to compute the eigenvalues of the linearization around the other singularities, but they will not be needed.
Instead, we just summarize the stability findings of singularities of \eqref{desrp} in \figref{eq}. In \figref{eqInS} we illustrate the reduced problem in the case $\frac{1-\frac{g}{3}}{\sqrt{1-\frac{2g}{3}}}<c<\sqrt{1-\frac{g}{3}}$,
where four folded singularities occur and where the regular singularity belongs to $C_S$ and is a saddle for the desingularized reduced problem \eqref{desrp}. The case $\sqrt{1-\frac{g}{3}}<c<v_s$ is similar except now the two non-symmetric folded singularities $f_3(c)$ and $f_4(c)$ occur on the $F_s$-branch located in the lower left corner.

\begin{figure}[!t]
        \centering
        \includegraphics[width=0.75\textwidth]{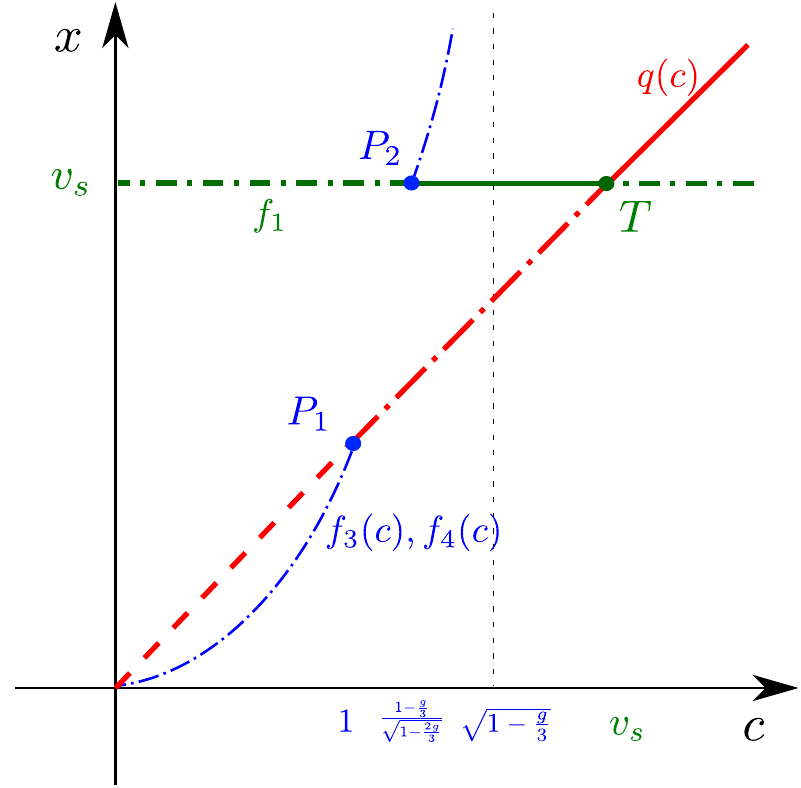}
        \caption{
        			Bifurcation diagram of singularities of \eqref{desrp} using $x=\frac12(v_1+v_2)$ on the vertical axis. Full lines indicate a stable node, dashed-dotted lines a saddle whereas dashed lines indicate an unstable node. In green, we indicate the folded singularity $f_1$ along $x=v_s$, $u=0$,
        whereas the red curve is the regular singularity $q=q(c)$ at $x=c$. The blue curves indicate the pairs \response{$f_3(c)$ and $f_4(c)$} of folded singularities\response{, see \propref{foldedsing} items \ref{cle1} and \ref{cge},} that bifurcate from $x=c$ and $x=v_s$ in the pitchfork bifurcations at $P_1$ and $P_2$ for $c=1$ and $c=\frac{1-\frac{g}{3}}{\sqrt{1-\frac{2g}{3}}}$, see \eqref{xu2fs}. The value $c=\sqrt{1-\frac{g}{3}}$ is an asymptote for the pair \response{$f_3(c)$ and $f_4(c)$} of folded singularities in blue. \response{For simplicity, the position of the pair $f_3(c)$ and $f_4(c)$ is not shown for $c>\sqrt{1-\frac{g}{3}}$.} 
            At $c=v_s$ there is a transcritical bifurcation $T$ where the folded singularity at $x=v_s$ exchanges stability with the regular singularity $x=c$.}
        \figlab{eq}
\end{figure}

%


\begin{figure}[!t]
        \centering
        \includegraphics[width=0.95\textwidth]{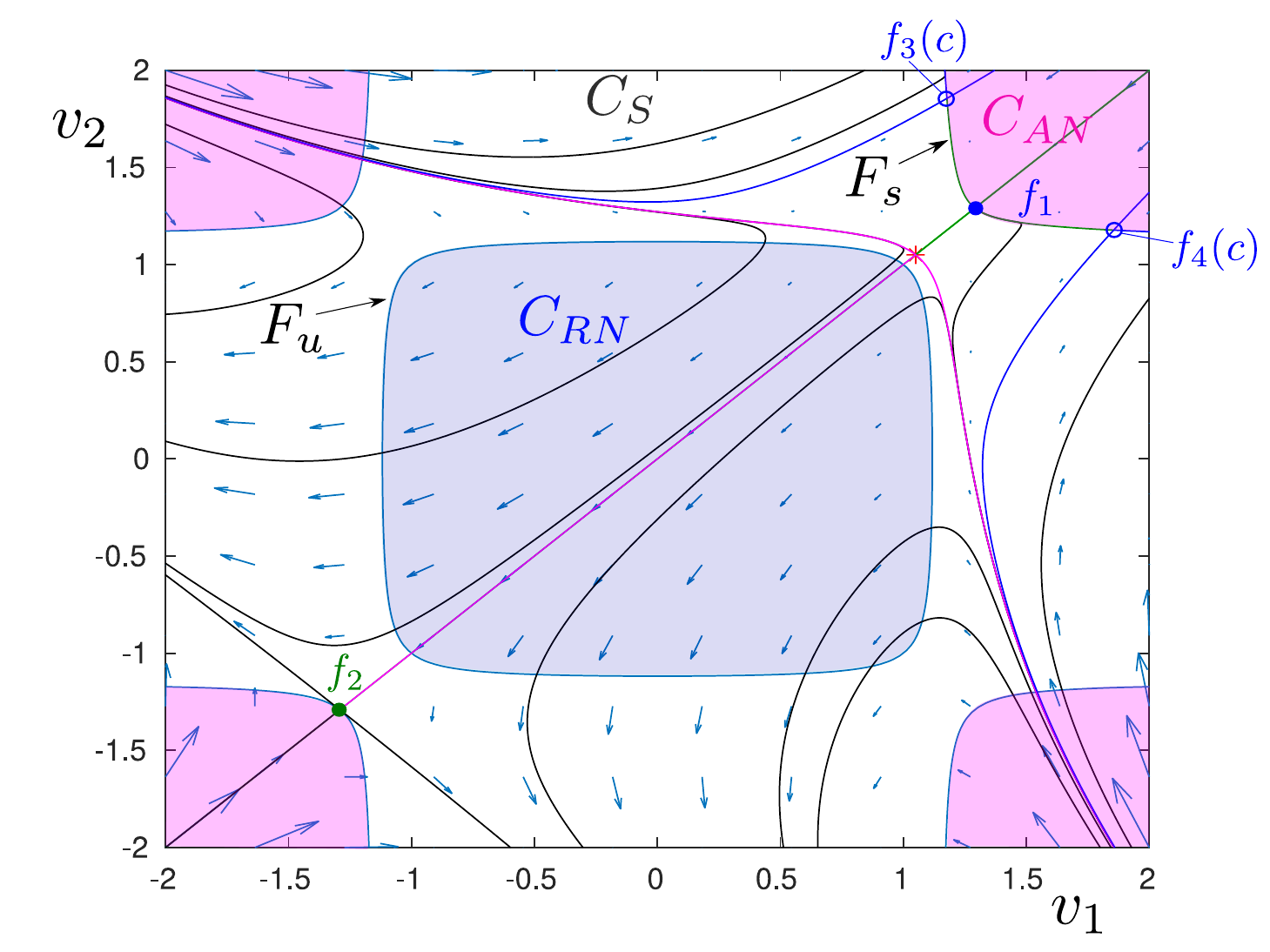}
        \caption{The phase portrait of the desingularized reduced problem \eqref{desrp} in the $v$-plane for $\frac{1-\frac{g}{3}}{\sqrt{1-\frac{2g}{3}}}<c<\sqrt{1-\frac{g}{3}}$. The plane is divided into $C_{RN}, C_S$ and $C_{AN}$ according to the stability of the critical manifold. The sets $F_u$ (bounded) and $F_s$ (unbounded) are the boundaries between $C_{RN}$ and $C_S$, respectively, $C_S$ and $C_{AN}$. There are four folded singularities $f_1,\ldots,f_4$ for these values of $c$. Two of these, $f_1$ and $f_2$, are indicated using two disks (blue and green) at $v=\pm \mathbf{v}_s$. The two blue circles indicate $f_3$ and $f_4$ which are saddles, with the blue curves being the stable manifolds. The red star is the regular singularity $q$ at $v=\mathbf{c}$, which is a saddle located 
        within $C_S$.  
        The case $\sqrt{1-\frac{g}{3}}< c<v_s$ is similar, but 
        {the nonsymmetric folded singularities $f_3$ and $f_4$ are then located on the $F_s$-branch shown in the lower left corner.} 
        }
        \figlab{eqInS}
\end{figure}
\subsection{Center manifold reduction}
We will now perform a center manifold reduction near $F_s$, which consist of partially hyperbolic points for $\epsilon=0$, recall \lemmaref{Fi}. In particular, the reduction will be based upon a local computation near the point $v=\mathbf{v}_s$, $w=\textbf w_s$, where $$w_s:=h_1(\mathbf{v}_s)=-v_s^2+3v_s.$$ (Recall the notation \eqref{tbfx}: $\textbf u=(u,u)$ for any $u$.) In further details,
we consider the extended system $(\eqref{fhn},\dot \epsilon=0)$.  Then $(v,w,\epsilon)=(\mathbf{v}_s,\textbf w_s,0)$ is partially hyperbolic, the linearization having one single nonzero eigenvalue given by $2g<0$. The associated eigenvector is $(1,1,0,0)$ i.e. along the ``symmetric fast'' space $v_1=v_2$. In terms of the center manifold computations, it is therefore useful to introduce 
\begin{align}
 x= \frac12(v_1+v_2),\quad u=\frac12 (v_1-v_2),\eqlab{xueqn}
\end{align}
so that $u=0$ corresponds to $v_1=v_2$ in which case we also have $x=v_1=v_2$. 
At the same time, it is also convinient to define define similar change of coordinates on the set of slow variables:\response{
\begin{align}
 y = \frac12 (w_1+w_2)-w_s,\quad z=\frac12 (w_1-w_2), \eqlab{yzeqn}
\end{align}
recall also \eqref{yuz}.}
This gives the following system:
\begin{equation}\eqlab{fhnx}
\begin{aligned}
 \dot x &= -x^3+3x-(y-w_s)-3xu^2,\\
 \dot u&=-z-u^3+3(v_s^2-x^2)u,\\
 \dot y&=\epsilon(x-c),\\
 \dot z &=\epsilon u,
\end{aligned}
\end{equation}
for which the symmetric subspace is now defined by $u=z=0$. In particular, the equations are now symmetric with respect to 
\begin{align}\eqlab{Sym}
\mathcal S:\quad (u,z)\mapsto (-u,-z),
\end{align} leaving $x$ and $y$ fixed. 
\response{The linearization of \eqref{fhnx} around $(v_s,0,w_s,0)$ for $\epsilon=0$ again leads to the nonzero eigenvalue $2g<0$, but now (by construction) the associated eigenvector is aligned with the $u$-axis. At the same time, we find a three-dimensional center space spanned by the vectors
\begin{align*}
 ((2g)^{-1},0,1,0)^T,(0,1,0,0)^T,(0,0,0,1)^T.
\end{align*}
}By center manifold theory, we obtain the following result.
\begin{proposition}\proplab{cmred}
There exists an attracting four dimensional symmetric (with respect to $\mathcal S$, see \eqref{Sym}) center manifold $M_a$ of the extended system $(\eqref{fhnx},\dot \epsilon=0)$ near $(x,u,y,z,\epsilon)=(v_s,0,0,0,0)$. It is locally a graph over $(u,y,z,\epsilon)$, i.e. there is a neighborhood $N$ of $(u,y,z,\epsilon) = (0,0,0,\response{0})$ such that 
\begin{align}
 M_a:\quad x &=v_s+\frac{1}{2g} y+\frac{3}{2g} v_s u^2+m(u,y,z,\epsilon),\quad (u,y,z,\epsilon)\in N,\eqlab{Maman}
 \end{align}
 where \response{the function $m:N\rightarrow \mathbb R$ is smooth and invariant with respect to $\mathcal S$}:
\begin{align*}
 m(u,y,z,\epsilon) = m(-u,y,-z,\epsilon),
\end{align*}
for all $(u, y,z,\epsilon)\in N$ (\response{where the right-hand is defined}),
and satisfies:
\begin{align*}
m(u, y,z,\epsilon):= \mathcal O(\epsilon, uz,u^4, y^2,z^2).
\end{align*}
\end{proposition}
\begin{proof}
 The existence of a (symmetric) center manifold follows from standard theory \cite{car1,haragus2011a}. The expansion \eqref{Maman} is also the result of a \response{direct} calculation.
\end{proof}

\subsection{The reduced \response{dynamics} on $M_a$}
We now proceed to study the reduced \response{dynamics} on $M_a$. 
For this, we insert \eqref{Maman} into \eqref{fhnx} and obtain
\begin{equation}\eqlab{cmred}
 \begin{aligned}
  \dot u &=-z-\frac{1}{g} \left(3v_s y +(9v_s^2+g) u^2 +n(u,y,z,\epsilon)\right)u,\\
  \dot y&=\epsilon \left(v_s-c+\frac{1}{2g} y+\frac{3v_s}{2g}  u^2+m(u, y,z,\epsilon)\right),\\
  \dot z &= \epsilon u,
 \end{aligned}
 \end{equation}
 on $M_a$. Here we have introduced a new smooth function \response{$n:N\rightarrow \mathbb R$} satisfying 
 \begin{align*}
  n(u,y,z,\epsilon)&=\mathcal O(\epsilon, uz,u^2  y, u^4, y^2,z^2).
 \end{align*} 
  The function $n$ is also invariant with respect to $\mathcal S$: $n(u, y,z,\epsilon)=n(-u,  y,-z,\epsilon)$ for all $(u,y,z,\epsilon)\in N$ (\response{where the right-hand side is well-defined}). 
  
  The system \eqref{cmred} is slow-fast with one fast variable $u$ and two slow variables $y$ and $z$. We first describe the layer problem associated with \eqref{cmred}:
  \begin{equation}\eqlab{cmredlay}
 \begin{aligned}
  \dot u &=-z-\frac{1}{g} \left(3v_s y +(9v_s^2+g) u^2 +n(u,y,z,0)\right)u,\\
  \dot{ y}&=0,\\
  \dot z &= 0,
 \end{aligned}
 \end{equation}
 We then have the following result.
 \begin{lemma}\lemmalab{cuspsing}
  The critical manifold $S$ of \eqref{cmredlay} is locally a graph over $u$, $y$:
  \begin{align}
  S: \quad z = Q(u,y),\eqlab{zQ}
 \end{align}
 where
 \begin{align*}
 Q(u,v)=-\frac{u}{g}\left(3v_s y+(9v_s^2+g)u^2+\mathcal O(y^2,u^2 y,u^4)\right),
 \end{align*}
 with $9v_s^2+g=9-5g>0$, see \eqref{vs}. The function $Q$ is smooth and odd in $u$: $Q(-u,y)=-Q(u,y)$ for all $u,y$ sufficiently small.
Moreover, locally $S$ splits into a disjoint union $S_a\cup F\cup S_r$ where 
 \begin{align}
  S_{r,a}:=S \cap \{y\gtrless f(u^2)\},\quad F = S\cap \{y=f(u^2)\},\eqlab{Sar}
 \end{align}
where
\begin{align}
 f(u^2):=-\frac{9v_s^2+g}{v_s}u^2+\mathcal O(u^4).\eqlab{Pu}
\end{align}
Finally, the point $(u,y,z)=\response{(0,0,0)}$ is a cusp singularity of $S$. 
 \end{lemma}
 \begin{proof}
  The result follows from the implicit function theorem. In particular, by implicit differentiation $S$ is non-normally hyperbolic at points where \response{$\frac{\partial Q}{\partial u}(u,y)=0$}; solving this equation, depending smoothly on $y$ and $u^2$, gives $y=f(u^2)$ with $f$ as in \eqref{Pu}, \response{or
    \begin{align}
     u^2 = f^{-1}(y)=-\frac{v_s}{9v_s^2+g}y+\mathcal O(y^2),\eqlab{u2finvy}
  \end{align}
locally by the implicit function theorem.} \response{Inserting \eqref{u2finvy} into $z=Q(u,y)$ gives
  \begin{align*}
   z^2 + a y^3 [1+\mathcal O(y)] =0,
  \end{align*}
  upon squaring both sides. Here $a= \frac{4v_s^3}{(9v_s^2+g)g^2}>0$. Setting $\bar y= y[1+\mathcal O(y)]^{1/3}$ and $\bar z= z/\sqrt{a}$ finally gives the cusp normal form $\bar z^2+\bar y^3=0$ \cite{arnold1984a}.}
 \end{proof}
Notice that $S$ is symmetric; $(u,y,z)\in S$ implies that $(-u,y,-z)\in S$ for all $(u, y,z)$ sufficiently small.  In particular, $(0, y,0)\in S$ for all $y\approx 0$. We illustrate the situation in \figref{MaRed}. 

\response{Any point $p\in W^s(S_a)$, belongs to the stable manifold of a base point on $S_a$. We shall denote this base point by 
\begin{align}
\pi_a(p)\in S_a,\eqlab{pia}
\end{align}
see \figref{MaRed}.}

\begin{figure}[!h]
        \centering
        \includegraphics[width=.6\textwidth]{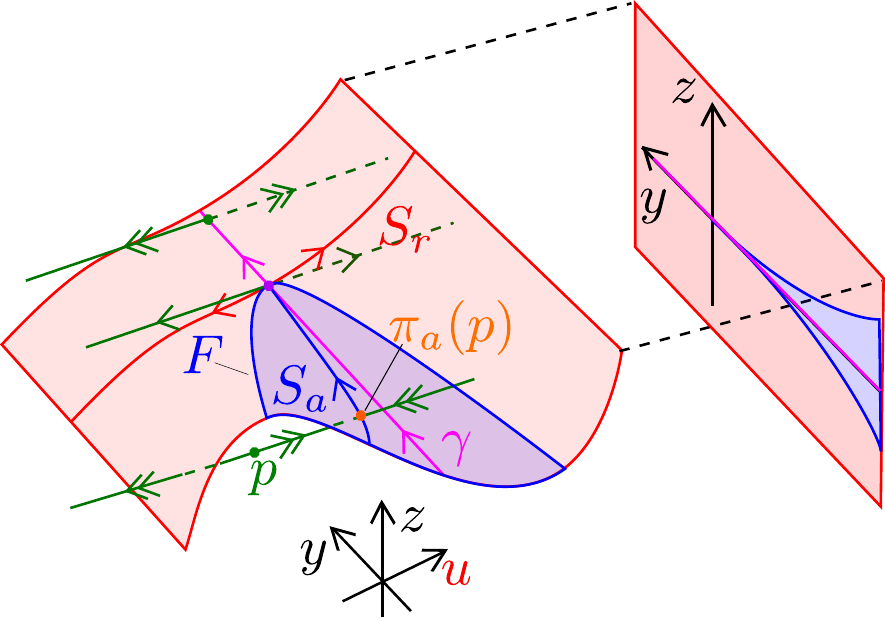}
        \caption{Sketch of the singular dynamics of \eqref{cmred} for $\epsilon=0$ and $c\approx v_s(g)$ but $c<v_s(g)$, illustrating how the critical manifold splits into a repelling sheet $S_r$ and an attracting sheet $S_a$ along the degenerate set $F$, see also \lemmaref{cuspsing}. The purple point indicates the cusp singularity (visible in the projection onto the $(y,z)$-plane), which acts like a node for the desingularized reduced problem on $S_a$. Due to the symmetry of the problem, the set $\gamma$ (in pink) given by $u=z=0$ is invariant for all $\epsilon>0$.}
        \figlab{MaRed}
%
\end{figure}
%

 
 Let
 \begin{align*}
  A(u,v) : = \begin{pmatrix}
  1 & -Q'_y(u,y) \\
  0 & Q'_u(u,y)
 \end{pmatrix}.
 \end{align*}
 \response{Here $Q'_s:=\frac{\partial Q}{\partial s}$ denotes the partial derivative of $Q$ with respect to $s=u,y$.}
Then the reduced problem on $S$ can be written as
 \begin{align}
  Q'_u(u,y)\begin{pmatrix}
  \dot u \\
  \dot y
 \end{pmatrix} &=A(u,v)\begin{pmatrix} u,\\
  v_s-c+\frac{1}{2g}y+\frac{3}{2g} v_s u^2+m(u,y,Q(u,y),0),\end{pmatrix}
 ,\eqlab{Sred}
 \end{align}
 by implicit differentiation. 
 \response{
 \begin{lemma}\lemmalab{topconj}
\eqref{Sred} is smoothly conjugated to the reduced problem \eqref{reducedC0} on $C$ in a neighborhood of  $v=\mathbf{v}_s$, $w=\textbf w_s$.
 \end{lemma}
 \begin{proof}
This is by construction: The critical $S$ within $M_a$ is the set of equilibria of \eqref{fhnx} and this set coincides with $C$ upon application of the coordinate transformation defined by \eqref{xueqn} and \eqref{yzeqn}. We therefore obtain the desired transformation through the $u$-equation in \eqref{xueqn} and the $y$-equation in \eqref{yzeqn}:
  \begin{align}
  u &= \frac12 (v_1-v_2),\quad   y= \frac12 (h_1(v_1,v_2)+h_2(v_1,v_2))-w_s.\eqlab{uyeqn}
  \end{align}
We see that $v=\mathbf{v}_s$ gives $(u,y)=(0,0)$ and the Jacobian matrix of the right hand sides with respect to $v=\mathbf{v}_s$ is 
\begin{align*}
 \begin{pmatrix}
  \frac12 & -\frac12 \\
  g & g
 \end{pmatrix}.
\end{align*}
Since this matrix is regular, having determinant $g<0$, \eqref{uyeqn} defines a diffeomorphism $(v_1,v_2)\mapsto (u,y)$ on a neighborhood of $v=\mathbf{v}_s$ by the inverse function theorem and this gives the desired conjugacy between \eqref{Sred} and \eqref{reducedC0}.
 \end{proof}}

\response{Consequently, our results on \eqref{reducedC0}, see e.g. \propref{foldedsing}, can (locally) be transferred to \eqref{Sred}. It will, however, be useful to perform the analysis of \eqref{Sred} in the $u,y$-plane directly nonetheless.}
To do so, we study a desingularization of \eqref{Sred}. Specifically, 
since $Q'_u<0$ on $S_a$, the system
 \begin{equation}\eqlab{Sred2}
 \begin{aligned}
 \begin{pmatrix}
  \dot u \\
  \dot y
 \end{pmatrix} &=-A(u,v)\begin{pmatrix} u,\\
  v_s-c+\frac{1}{2g}y+\frac{3}{2g} v_s u^2+m(u,y,Q(u,y),0),\end{pmatrix}
 \end{aligned}
 \end{equation}
 is equivalent to \eqref{Sred} there. Orbits of \eqref{Sred} on $S_r$ are also orbits of the desingularized system \eqref{Sred2} but the direction of the flow has changed. $(u,y)=(0,0)$ is then an equilibrium of the desingularized system \eqref{Sred2}, and a \response{direct} calculation shows that the eigenvalues of the linearization are $(-2g) \lambda_1, (-2g) \lambda_2,$ recall \eqref{eigval}. Therefore for all $c$ in the interval $\frac{1-\frac{g}{3}}{\sqrt{1-\frac{2g}{3}}}<c<v_s$, $(u,y)=(0,0)$ is a hyperbolic stable node for \eqref{Sred}, recall also \propref{foldedsing}. In particular, using \eqref{eigval} we find \response{that} $\lambda_1=\lambda_2<0$ for $c=\frac{1-\frac12 g}{\sqrt{1-\frac{2g}{3}}}$. This value of $c$ is always less than $v_s$ for $g<0$ (and  greater than the value $\frac{1-\frac13 g}{\sqrt{1-\frac{2g}{3}}}$ corresponding to the second pitchfork bifurcation, recall \propref{foldedsing}). A direct calculation then gives the following.
 \begin{lemma}\lemmalab{cinterval}
  Consider
  \begin{align}
  c\in \left(\frac{1-\frac12 g}{\sqrt{1-\frac{2g}{3}}},v_s\right).\eqlab{cinterval}
  \end{align}
Then 
\begin{align}
 \lambda_2<\lambda_1<0,\eqlab{ratio}
\end{align}
and the invariant set $u=0$ is therefore the weak direction of the stable node $(u,y)=(0,0)$.  
 \end{lemma}

Consequently, all points on $S_a$ approaches $(u,y)=(0,0)$ tangentially to the set $u=0$ under the forward flow of \eqref{Sred2} for these values of $c$.
 In the following, we shall denote {the corresponding set $(0,y,0)$ in the $(u,y,z)$ space by $\gamma$}. For $c$ as in \eqref{cinterval}, it corresponds to a singular weak canard for the folded node \cite{wechselberger_existence_2005}. In fact, $\gamma$ is distinguished from all orbits on $S_a$ insofar that it is symmetric with respect to $\mathcal S$.

As $S$ has a cusp singularity at $(u,y,z)=0$, $v=\mathbf{v}_s,w=h(\mathbf{v}_s)$ is not a regular folded node singularity \cite{wechselberger_existence_2005} of the slow-fast system \eqref{fhn}. We refer to it as a \textbf{cusped node}.

 Notice also that in the nonhyperbolic case $c=v_s$, it follows from $Q(0,y)=0$ that the invariant set $\gamma$ becomes an attracting center manifold of \eqref{Sred2} along which we have $\dot y = \frac{3v_s}{2g^2}y^2(1+ \mathcal O(y))>0$ on $S_a$.  Despite the resemblance, the transcritical bifurcation at $c=v_s$ is also not a folded saddle-node (type II) \cite{krupa2010a}. We will instead call it a \textbf{cusped saddle-node}.

 In the following, we describe the dynamics of \eqref{cmred} near the cusped node $(u,y,z)=0$ for $\epsilon=0$ (corresponding to $v=\mathbf{v}_s$ in \figref{eqInS}) for $c$-values fixed in the interval \eqref{cinterval}. Here we will blowup $(u,y,z,\epsilon)=0$ and describe how trajectories that start near $S_a$ will evolve as the pass the folded singularity. Subsequently, we will turn our attention to the {cusped saddle-node}. For this purpose, we will (essentially) include $c$ in the blowup transformation and blowup $(u,y,z,\epsilon,c)=\response{(0,0,0,0,v_s)}$.  This will enable us to describe the onset and termination of MMOs.

 \begin{remark}\remlab{strong}
  \response{The point $(u,y)=(0,0)$ also has a strong eigendirection for \eqref{Sred2} along $y=0$ whenever \eqref{cinterval} holds. In fact, a direct calculations shows that
 $\dot u<0$ along the fold $y=f(u^2)$, $u\ne 0$, and, as a consequence, the strong stable manifold for \eqref{Sred2} lies completely within the repelling subset $S_r$ of $S$ in this case, see the red orbit on $S$ in \figref{MaRed}. Therefore, since the direction is reversed on $S_r$, it follows that \eqref{Sred} does not have a strong canard (in contrast to the standard folded node).

 The lack of a strong canard relates to another important difference between the folded node and the cusped node. Indeed, for the folded node, there is only one fast direction away from the fold. In contrast, as we see in \figref{MaRed}, there are two separate fast directions (in green, using single headed arrows to indicate the lack of hyperbolicity) away from the cusp. 
 }
 \end{remark}
 
 \section{Blowup analysis of the cusped node}\label{sec:4}
 
Consider the extended system obtained from augmenting \eqref{cmred} by $\dot \epsilon=0$ for the parameter values \eqref{cinterval} and denote the resulting right hand side by $V(u,y,z,\epsilon,c)$. Then $(u,y,z,\epsilon)=0$ is a degenerate equilibrium of the vector-field $V$, with the linearization having only zero eigenvalues.
We therefore perform a spherical blowup transformation \cite{dumortier1996a,szmolyan_canards_2001} of $(u,y,z,\epsilon)=0$:
\begin{align}
 \Psi:\quad (r,(\bar u,\bar y,\bar z,\bar \epsilon))\mapsto \begin{cases}
                                                  u &=r\bar u,\\
                                                  y&=r^2\bar y,\\
                                                  z &=r^3 \bar z,\\
                                                  \epsilon &=r^4\bar \epsilon,
                                                 \end{cases}\eqlab{blowup1}
\end{align}
with $r\in [0,r_0]$, $r_0>0$ small enough,  $(\bar u,\bar y,\bar z,\bar \epsilon)\in S^3$ where
\begin{align*}
 S^3 = \left\{(x_1,x_2,x_3,x_4)\in \mathbb R^4\,:\,\sum_{i=1}^4 x_i^2=1\right\},
\end{align*}
is the unit $3$-sphere. In this way, the degenerate point $(u,y,z,\epsilon)=0$ gets blown up through the preimage of \eqref{blowup1} to the $3$-sphere with $r=0$. Let $\overline V=\Psi^* V$ denote the pull-back of $V$ under \eqref{blowup1}. Then the exponents on $r$ (also called weights) in \eqref{blowup1} have been chosen so that 
\begin{align}
\widehat V := r^{-2} \overline V,\eqlab{hatV}
\end{align}
is well-defined and non-trivial for $r=0$. $\widehat V$, being equivalent with $V$ for $r>0$, will have improved hyperbolicity properties for $r=0$ and it is therefore this vector-field that we will study in the following. To do so we will use certain directional charts \cite{krupa_extending_2001}. \response{We will focus on two charts}: the ``entry chart'' obtained by setting $\bar y=-1$ in \eqref{blowup1}, and the ``scaling chart'' obtained by setting $\bar \epsilon=1$ in \eqref{blowup1}. That is, we consider local coordinates $(r_1,u_1,z_1, \epsilon_1)$ and $(r_2,u_2,y_2,z_2)$, parametrizing the subset of the sphere where $\bar y<0$ and where $\bar \epsilon>0$, respectively, such that \eqref{blowup1} takes the following local forms:
\begin{align}
 (r_1,u_1,z_1,\epsilon_1)\mapsto \begin{cases}
  u &= r_1u_1,\\
  y&=-r_1^2,\\
  z &=r_1^3 z_1,\\
  \epsilon &=r_1^4\epsilon_1.
 \end{cases}\eqlab{blowup11}
\end{align}
and
\begin{align}
 (r_2,u_2,y_2,z_2)\mapsto \begin{cases}
  u &= r_2u_2,\\
  y&=r_2^2 y_2,\\
  z &=r_2^3 z_2,\\
  \epsilon &=r_2^4,
 \end{cases}\eqlab{blowup12}
\end{align}
respectively. 
We will refer to these chart\response{s} as $\bar y=-1$ and $\bar \epsilon=1$ in the following and the dynamics in each of these are analyzed in the following sections. Notice that the charts \eqref{blowup11} and \eqref{blowup12} overlap on $\bar y<0,\bar \epsilon>0$ and the associated \response{change of coordinates} is given by the following expressions:
\begin{align}\eqlab{cc1}
 r_2 = r_1 \epsilon_1^{1/4},\quad z_2 = z_1 \epsilon_1^{-3/4},\quad y_2 = -\epsilon_1^{-1/2},\quad u_2 = u_1 \epsilon_1^{-1/4},
\end{align}
for $\epsilon_1>0$. In the following, we analyze the dynamics in each of the two charts. 
\response{The analysis of the remaining charts, required to cover the sphere completely, is similar and therefore left out.} We summarize our findings in \figref{blowup2}. We refer to the figure caption for further details. In the following, we will use the convention that a set, say $P$, will be given a subscript $1$ or $2$ when expressed in the respective charts $\bar y=1$ and $\bar \epsilon=1$. When the charts overlap, $P_1$ will then be related by $P_2$ under the change of coordinates \eqref{cc1}. 
\begin{remark}\remlab{broercusp}
 The references \cite{broer2013a,jard2016a} also describe a slow-fast cusp singularity in $\mathbb R^3$ using GSPT and blowup. However, their blowup weights differ from ours since these references consider the cusp in absence of singularities of the reduced flow. The results of \cite{broer2013a,jard2016a} therefore generalizes \cite{szmolyan2004a} on regular jump points. Moreover, at the level of the layer problem our setting corresponds to a time reversal of the system in \cite{broer2013a,jard2016a}, i.e. their $S_{a,r}$ correspond to our $S_{r,a}$, respectively. 
\end{remark}
\begin{figure}[!t]
        \centering
        \includegraphics[width=0.65\textwidth]{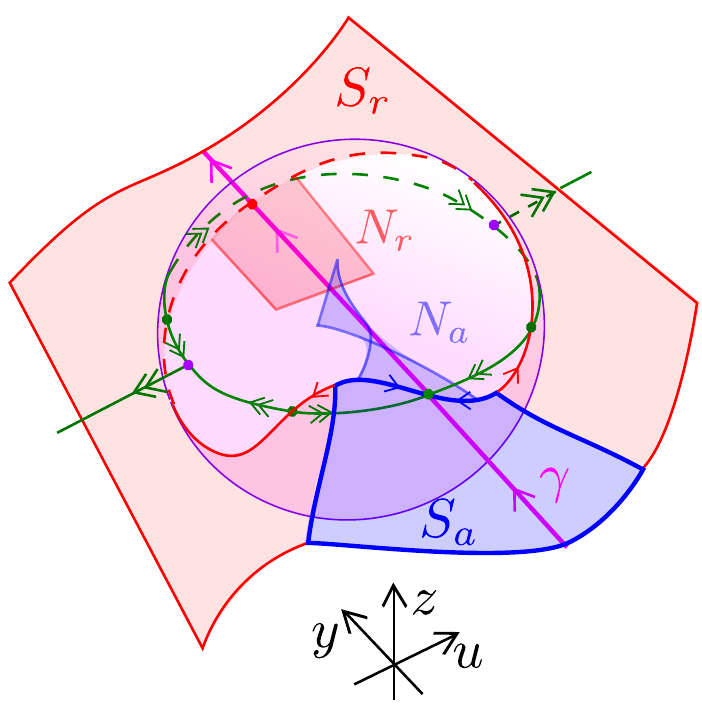}
        \caption{Illustration of the spherical blowup of the cusped node. The blowup transformation \eqref{blowup1} allows us to extend subsets of the critical manifolds $S_{a,r}$ onto the sphere $S^3$ as invariant manifolds $N_{a,r}$ of a desingularized vector-field. Since $\bar \epsilon\ge 0$, we illustrate the resulting hemi-sphere as a solid sphere (shaded and purple) with $\bar \epsilon>0$ inside. As indicated, these extended manifolds, which lie inside, intersect transversally along $\gamma$ in general (when the ratio $\frac{\lambda_2}{\lambda_1}$ of the eigenvalues is not an integer, see \lemmaref{twist}) and the number of twists of $N_a$ and $N_r$ along $\gamma$ can, as in the folded node, be directly related to the number of SAOs, see \thmref{main1}. }
        \figlab{blowup2}
\end{figure}

\subsection{Analysis in the $\bar y=-1$-chart}
Inserting \eqref{blowup11} into \eqref{cmred} with $\dot \epsilon=0$ augmented gives
\begin{equation}\eqlab{hatV1}
\begin{aligned}
\dot r_1 &=-\frac{1}{2} r_1\epsilon_1 \left[v_s-c+r_1^2\left(-\frac{1}{2g}+\frac{3v_s}{2g}  u_1^2+ \response{\mathcal O(r_1^2)}\right)\right],\\
\dot u_1 &=-z_1-\frac{1}{g}\left(-3v_s +(9v_s^2+g)u_1^2+\response{\mathcal O(r_1^2)}\right)u_1\\
&+\frac{ 1}{2}u_1\epsilon_1 \left[v_s-c+r_1^2\left(-\frac{1}{2g}+\frac{3v_s}{2g}  u_1^2+ \response{\mathcal O(r_1^2)}\right)\right],\\
\dot z_1 &=\epsilon_1 \left(u_1 +\frac{3}{2}z_1\left[v_s-c+r_1^2\left(-\frac{1}{2g}+\frac{3v_s}{2g}  u_1^2+ \response{\mathcal O(r_1^2)}\right)\right]\right),\\
\dot \epsilon_1 &=2 \epsilon_1^2 \left[v_s-c+r_1^2\left(-\frac{1}{2g}+\frac{3v_s}{2g}  u_1^2+ \response{\mathcal O(r_1^2)}\right)\right],
\end{aligned}
\end{equation}
after division of the right hand side by $r_1^2$. All $\mathcal O$-terms are smooth functions. This is our local form of $\widehat V$, recall \eqref{hatV}. The set $r_1=\epsilon_1=0$ is invariant and on this set we find that $\dot z_1=0$ and 
\begin{align*}
 \dot u_1 &=-z_1-\frac{1}{g}\left(-3v_s +(9v_s^2+g)u_1^2\right)u_1.
\end{align*}
There is therefore a critical manifold $S_1$ along $r_1=\epsilon_1=0$ given by 
\begin{align}
 z_1 = -\frac{1}{g}\left(-3v_s +(9v_s^2+g)u_1^2\right)u_1.\eqlab{z1u1}
\end{align}
It is the critical manifold $S$ extended to the blowup sphere, where it has improved hyperbolicity properties. In particular, let 
\begin{align}\eqlab{u1p}
u_{p,1}:= \sqrt{\frac{v_s}{9 v_s^2+g}}.
\end{align} Then the subset $S_{a,1}$ of $S_1$ within $u_1\in (-u_{p,1},u_{p,1})$ is partially attracting, the linearization about any point in this set $(r_1,u_1,z_1,\epsilon_1)\in S_{a,1}$ having one single nonzero and negative eigenvalue. Consequently, by center manifold theory we obtain the following result. 
\begin{proposition}\proplab{N1a}
For any $\nu>0$ small enough, consider $I(\nu):=[-u_{p,1}+\nu,u_{p,1}-\nu]$. Then there exists a three dimensional center manifold $N_{a,1}$ of points $(0,u_1,0,0)$, $u_1\in I(\nu)$ of the following graph form:
\begin{align}
 N_{a,1}:\quad z_1 = u_1\left(-\frac{1}{g}\left(-3v_s +(9v_s^2+g)u_1^2\right)+\mathcal O(r_1^2,\epsilon_1)\right),\eqlab{Na1}
\end{align}
for $u_1\in I(\nu),\,(r_1,\epsilon_1)\in [0,\delta]^2$ and some $\delta>0$ small enough.
\end{proposition}
In the expansion \eqref{Na1}, we have used that $u_1=z_1=0$ is invariant for all $r_1,\epsilon_1\ge 0$.

As usual, $N_{a,1}$ is foliated by constant $\epsilon$-values: $\epsilon=r_1^4\epsilon_1$ and $N_{a,1}\cap \{\epsilon=r_1^4\epsilon_1\}$ therefore provides an extension of Fenichel's slow manifold $S_{a,\epsilon}$, being a perturbation of a compact subset of $S_a$, up close to the blowup sphere. Specifically, at $\epsilon_1=\delta$ we have $r_1= \epsilon^{1/4}\delta^{-1/4}=\mathcal O(\epsilon^{1/4})$ and $N_{a,1}\cap\{\epsilon=r_1^4\epsilon_1\}$, upon blowing down to the $(u,y,z)$-variables, therefore extends $S_{a,\epsilon}$ as an invariant manifold up to a ``wedge-shaped'' region of $(0,0,0)$ which extends $\mathcal O(\epsilon^{1/4}),\mathcal O(\epsilon^{1/2}),\mathcal O(\epsilon^{3/4})$ in the $u,y,z$-directions, respectively, recall \eqref{blowup1}.

A \response{direct} calculation shows that the reduced problem on $N_{a,1}$ is given by 
\begin{equation}\eqlab{redN1a}
\begin{aligned}
 \dot r_1 &=-\frac12 r_1,\\
 \dot u_1 &=u_1 \left(-\frac{\lambda_2}{\lambda_1}+\frac12 +\mathcal O(u_1^2,r_1^2,\epsilon_1)\right),\\
 \dot \epsilon_1 &=2\epsilon_1,
\end{aligned}
\end{equation}
recall \eqref{eigval}. 
Here we have used a desingularization through division by $$\epsilon_1 \left[v_s-c+\mathcal O(r_1^2)\right];$$ notice that the \response{square} bracket is positive for any $r_1\ge 0$ small enough by assumption of \eqref{cinterval}. Then for $c$ as in \eqref{cinterval}, see also \eqref{ratio}, one can show that $(r_1,u_1,\epsilon_1)=0$ is the only equilibrium on $N_{a,1}$ and it is hyperbolic for \eqref{redN1a} with eigenvalues 
\begin{align}\eqlab{eig1}
-\frac12,-\frac{\lambda_2}{\lambda_1}+\frac12,2.
\end{align}
\begin{lemma}\lemmalab{linearization}
 \response{ Suppose that \eqref{ratio} holds and that
 \begin{align*}
  \frac{\lambda_2}{\lambda_1} \ne 3.
 \end{align*}
Then there is a $C^1$-linearization of \eqref{redN1a} of the form:
 \begin{align}
  (r_1,u_1,\epsilon_1) \mapsto \tilde u = \psi(r_1,u_1,\epsilon_1),\eqlab{tildeu1}
 \end{align}
 where $\psi(0,0,0)=0,\frac{\partial \psi}{\partial u_1}(0,0,0)=1$, so that
 \begin{align*}
  \dot{\tilde u}_1 &= \tilde u_1 \left(-\frac{\lambda_2}{\lambda_1}+\frac12\right).
 \end{align*}
  } 
\end{lemma}
\begin{proof}
\response{
According to the classical work \cite{Belitskii1973}, a smooth system $\dot x = Ax+\mathcal O(x^2)$, with eigenvalues $\nu_i$, $i=1,\ldots,n$, of the matrix $A\in \mathbb R^{n\times n}$, is linearizable by a $C^1$-diffeomorphism if 
\begin{align*}
 \nu_i \ne \operatorname{Re} (\nu_j+\nu_k),
\end{align*}
for all $i=1,\ldots,n$ and all $\operatorname{Re} \nu_j<0$ and $\operatorname{Re} \nu_k>0$. In the present case, this gives
\begin{align*}
-\frac{\lambda_2}{\lambda_1}+\frac12 \ne - \frac12 +2=\frac32,
\end{align*}
setting $\nu_i= -\frac{\lambda_2}{\lambda_1}+\frac12$, $\nu_j = -\frac12$ and $\nu_k=2$,
and
\begin{align*}
 -\frac12 \ne -\frac{\lambda_2}{\lambda_1}+\frac12 + 2 = -\frac{\lambda_2}{\lambda_1}+\frac52,
\end{align*}
setting $\nu_i= -\frac12$, $\nu_j=-\frac{\lambda_2}{\lambda_1}+\frac12$, and $\nu_k=2$, recall \eqref{eig1}.
The first inequality clearly holds since the left hand side is negative by \eqref{ratio}. Similar, the second inequality implies $\lambda_2\ne 3 \lambda_1$ and the existence of the $C^1$-linearization therefore follows. Seeing that the $r_1$- and $\epsilon_1$-equations are already linear, one can easily show that the linearization takes the form \eqref{tildeu1}. }
\end{proof}

\subsection{Analysis in the $\bar \epsilon=1$-chart}
Inserting \eqref{blowup12} into \eqref{cmred} with $\dot \epsilon=0$ augmented gives
\begin{equation}\eqlab{hatV2}
\begin{aligned}
 \dot u_2 &=-z_2-\frac{1}{g} \left(3v_s  y_2 +(9v_s^2+g) u_2^2 +\mathcal O(r_2^2)\right)u_2,\\
 \dot y_2 &= v_s-c+r_2^2\left(\frac{1}{2g}y_2+\frac{3v_s}{2g}  u_2^2+ \mathcal O(r_2^2)\right),\\
 \dot z_2 &=u_2,
\end{aligned}
\end{equation}
and $\dot r_2=0$, 
upon division of the right hand side by the common factor $r_2^2$. All $\mathcal O$-terms are smooth. This is our local form of $\widehat V$, recall \eqref{hatV}. Notice that since $\gamma:\,u=z=0$ is invariant for all $\epsilon\ge 0$, the set $\gamma_2$ defined by $u_2=z_2=0$ is also invariant for all $r_2\ge 0$ and $y_2$ increases for $c$ in the interval \eqref{cinterval} \response{for all $0<r_2\ll 1$}. Consider $r_2=0$. Then we obtain 
\begin{equation*}
\begin{aligned}
\frac{du_2}{dy_2} &=\frac{1}{v_s-c}\left(-z_2-\frac{1}{g} \left(3v_s  y_2 +(9v_s^2+g) u_2^2\right)u_2\right),\\
 \frac{dz_2}{dy_2} &=\frac{u_2}{v_s-c}.
\end{aligned}
\end{equation*}
Linearization around $u_2=z_2=0$ gives \response{
\begin{equation}\eqlab{U2Z2var}
\begin{aligned}
\frac{dU_2}{dy_2} &=\frac{1}{v_s-c}\left(-Z_2-\frac{3v_s}{g}y_2U_2\right),\\
 \frac{dZ_2}{dy_2} &=\frac{U_2}{v_s-c}.
\end{aligned}
\end{equation}}
Setting 
\response{\begin{align*}
 y_2 = \sqrt{\frac{-g(v_s-c)}{3v_s}}Y_2,
\end{align*}}
we can write this system as a Weber equation:
\response{\begin{align}
 U_2''(Y_2) - Y_2 U_2'(Y_2) + \frac{\lambda_2}{\lambda_1} U_2(Y_2) =0,\eqlab{weber}
\end{align}
recall \eqref{eigval}.} The implication of this is the following: Let $N_{a,2}(r_2)$ denote the center manifold obtained in the chart $\bar y=-1$ written in the $\bar \epsilon=1$\response{-coordinates} $(u_2,y_2,z_2,r_2)$. It is parametrized by $r_2=\epsilon^{1/4}$ and $N_{a,2}(0)$ denotes the intersection with $r_2=0$, i.e. with the blowup sphere. 

Working in the $\bar y=1$ chart, for example, we may also obtain a repelling critical manifold $N_{2,r}(r_2)$ in much the same way. This extends the repelling slow manifold $S_{r,\epsilon}$ into scaling chart as an invariant manifold and we write $N_{2,r}(0)$ to denote the intersection with $r_2=0$. We extend each of these manifolds by the flow and denote the extended objects by the same symbol. Then $\gamma_2\subset N_{a,2}(0)\cap N_{2,r}(0)$. Using \eqref{weber}, we have the following.
\begin{lemma}\lemmalab{twist}
 The intersection of $N_{a,2}(0)$ and $N_{2,r}(0)$ along $\gamma_2$ is transverse whenever $\frac{\lambda_2}{\lambda_1}\notin \mathbb N$. In the affirmative case, the tangent space of $N_{a,2}(0)$ along $\gamma_2$ twists $\lfloor \frac{\lambda_2}{\lambda_1}\rfloor$-many times, where each twist corresponds to a full rotation by $180^\circ$ degrees. 
\end{lemma}
\begin{proof}
 The proof of this is identical to the proof of \cite[Lemma 4.4]{szmolyan_canards_2001} for the folded node. \response{Basically, regarding the transversality, 
 we first use that the tangent spaces of $N_{a,2}(0)$ and $N_{2,r}(0)$ along $\gamma_2$ coincide with the set of solutions of \eqref{U2Z2var} having algebraic growth as $y_2\rightarrow \mp \infty$, respectively.} Next, 
for $\frac{\lambda_2}{\lambda_1}\notin \mathbb N$ it is standard that there are no bounded solutions of \eqref{weber}. \response{This proves the transversality.  
Finally, regarding the number of twists, we use that any solution of \eqref{weber} having algebraic growth as $y_2\rightarrow -\infty$ has $\lfloor \frac{\lambda_2}{\lambda_1}\rfloor+1$ simple zeros.  Two consecutive zeros correspond to a full $180^\circ$-rotation in the $(U_2,Z_2)$-plane and the result therefore follows. }
%
\end{proof}
\begin{remark}\response{
 Whenever $n:=\frac{\lambda_2}{\lambda_1}\in \mathbb N$, then $U_2(Y_2)=H_n(Y_2/\sqrt{2})$, with $H_n$ the Hermite polynomial of degree $n$, is a bounded solution of \eqref{weber}, see \cite{wechselberger_existence_2005}. This means that $TN_{a,2}(0)=TN_{r,2}(0)$ and the tangent spaces form a single band with $\frac{\lambda_2}{\lambda_1}$-many twists.} This may give rise to secondary intersections of $N_{a,2}$ and $N_{a,r}$ (like secondary canards, see \cite{kristiansen2020a,wechselberger_existence_2005}) upon perturbation. But in contrast to the folded node, the bifurcations $\frac{\lambda_2}{\lambda_1}\in \mathbb N$ do not produce additional intersections of the Fenichel slow manifolds themselves, since in our case we do not have a strong canard, \response{recall \remref{strong}}. See also \cite{kristiansen2020a}. We therefore do not pursue the description of these bifurcations any further.
\end{remark}

\subsection{Completing the analysis of the cusped node }
We can now state our main results on the dynamics for fixed $c$ in the interval \eqref{cinterval}. Firstly, following \lemmaref{twist} and the fact that $N_{a,2}(r_2)$ and $N_{2,r}(r_2)$ are $\mathcal O(r_2^2)$-close to $N_{a,2}(0)$ and $N_{2,r}(0)$ in the scaling chart, we conclude:
\begin{proposition}
The Fenichel slow manifolds $S_{a,\epsilon}$ and $S_{r,\epsilon}$ intersect transversally along $\gamma$ whenever $\lfloor \frac{\lambda_2}{\lambda_1}\rfloor\notin \mathbb N$ \response{for all $0<\epsilon\ll 1$}.
\end{proposition}
\begin{remark}
A similar result holds for the folded node, see e.g. \cite{szmolyan_canards_2001,wechselberger_existence_2005}. But in contrast to these results, we are here deliberately referring to the Fenichel slow manifolds, i.e. the slow manifolds obtained from perturbing compact subsets of $S_a$ and $S_r$ through Fenichel's theory \cite{fen3} and extending these by the forward flow. For the general folded node, it is only invariant manifolds -- that have been extended as center-like manifolds -- that are shown to intersect transversally along a weak canard; the Fenichel slow manifolds are only a subset of these extended manifolds. This relates to the delicacy of the weak canard and whether this object in fact ever reaches the Fenichel slow manifolds, see also \cite{kristiansen2020a} for a discussion of these technical aspects. The reason why we can be more specific in the present context is that $\gamma$, which plays the role of the weak canard, exits for all $\epsilon>0$ for our system and  this set therefore (locally) belongs to $S_{a,\epsilon}$ and $S_{r,\epsilon}$. 
\end{remark}

\response{We now proceed to state our main result on the SAOs of the cusped node. For this, we will follow \cite{wechselberger_existence_2005} and count, in line with \lemmaref{twist}, the number of SAOs as the number of full $180^\circ$-rotations in a plane transverse to $\gamma$. More precisely, consider an orbit $O:t\mapsto (u(t),y(t),z(t))$, $t\in I:=[0,T]$, with $$(u(t),z(t))\ne (0,0),$$ for all $t\in I$. The number of SAOs is then the rotation number
\begin{align*}
n = \lfloor (\Theta(T)-\Theta(0))/\pi\rfloor,
\end{align*}
where $\Theta(t)\in \mathbb R$, $t\in [0,T]$, is the lift of the angle
$\theta(t)\in \mathbb R/ 2\pi \mathbb Z$ defined by $\tan \phi(t)= \frac{z(t)}{u(t)}$, $t\in [0,T]$.  
Similarly, we define the amplitude of the SAOs as $\max_{t\in [0,T]}\vert (u(t),z(t))\vert$. (In \thmref{main2}, however, we will measure the amplitude in terms of $\vert (u_2(t),z_2(t))\vert$).} 

\response{In the following, we write $f=f(\mu)\sim \mu$ whenever there are positive constants $c_1<c_2$ such that 
\begin{align*}
 c_1 \mu \le f(\mu) \le c_2 \mu,
\end{align*}
for all $0<\mu\ll 1$.
}


\begin{theorem}\thmlab{main1}
\response{Fix $c$ as in \eqref{cinterval}, any $\delta>0$ sufficiently small and suppose that 
\begin{align*}
 \frac{\lambda_2}{\lambda_1}\notin \mathbb N.
\end{align*}
Consider any point $p$ so that $\pi_a(p)\in S_a\backslash \gamma$, recall \eqref{pia}. 
Then the following holds for all $0<\epsilon\ll 1$: The forward orbit of $p$ intersects the section defined by $y=y_{\mathrm{exit}}:=(\epsilon\delta^{-1})^{1/2}$ in a point $(u,y,z)=(u_{\mathrm{exit}},y_{\mathrm{exit}},z_{\mathrm{exit}})$ with 
\begin{align}
 u_{\mathrm{exit}} \sim \epsilon^{\frac{\lambda_2}{2\lambda_1}},\quad z_{\mathrm{exit}} \sim \epsilon^{\frac12+\frac{\lambda_2}{2\lambda_1}},\eqlab{amplitudeuz}
\end{align}
and undergoes $\lfloor \frac{\lambda_2}{\lambda_1}\rfloor$ many SAOs. The order of the amplitude of the SAOs are given by \eqref{amplitudeuz}.}
 \end{theorem}
\begin{proof}
 We first work in the $\bar y=-1$ chart. Then the forward flow of the point $p$ can be described by the reduced problem \eqref{redN1a}. We therefore integrate these equations from $(r_{10},u_{10},\epsilon_{10})$ to $(r_{11},u_{11},\epsilon_{11})$ with $\epsilon_{11}=\delta>0$. \response{Here $r_{10},u_{10}=\mathcal O(1)$ and $\epsilon_{10}\sim \epsilon$ as $\epsilon\rightarrow 0$ by assumption on $\pi_a(p)\notin \gamma$}. To perform the integration, we apply  \lemmaref{linearization} and consider
 \begin{align*}
  \dot r_1 &= -\frac12 r_1,\\
  \dot{\tilde u}_1 &= \tilde u_1\left(-\frac{\lambda_2}{\lambda_1}+\frac12\right),\\
  \dot \epsilon_1 &=2 \epsilon_1.
 \end{align*}
\response{Integrating these equations gives
\begin{align*}
 \tilde u_{11} = \left(\epsilon_{11}^{-1} \epsilon_{10}\right)^{\frac{\lambda_2}{2\lambda_1}-\frac14 }\tilde u_{10},
\end{align*}
and $r_{11} = (\epsilon\delta^{-1})^{1/4}$. Therefore $u_{11} \sim \epsilon^{\frac{\lambda_2}{2\lambda_1}-\frac14 }$ using \eqref{tildeu1}. 
We then transform the result using \eqref{cc1} to the scaling chart. Here we apply regular perturbation theory from $y_2=-\delta^{-1/3}$, which corresponds to $\epsilon_{1}=\delta$, up to $y_2=\delta^{-1/3}$. This value of $y_2$ corresponds to $y_{\mathrm{exit}}$. Now, the order of the amplitude of $u_2$ and $z_2$ does not change during \response{this} finite time passage. Using \eqref{Na1} and \eqref{blowup11}, we therefore finally obtain \eqref{amplitudeuz}. The number of small amplitude oscillations follow from \lemmaref{twist} upon taking $\delta>0$ small enough (and subsequently $\epsilon>0$ small enough).}
\end{proof}

\section{Analysis of the cusped saddle-node}\label{sec:5}
Next, we consider the cusped saddle-node where $c\approx v_s$ in \eqref{cmred}. \response{For this we will use an $\epsilon$-dependent zoom near $v_s$. Looking at \eqref{hatV2} with $r_2=\epsilon^{1/4}$, we see that 
\begin{align}
c = v_s + \sqrt{\epsilon} c_2,\eqlab{c2expr}
\end{align}
brings the two terms in the equation for $y_2$ to the same order. For this reason, we now consider \eqref{c2expr}
before applying the blowup transformation $\Phi$.} In this way, $c=v_s$ gets blown up to $c_2\in \mathbb R$ for $\epsilon=0$. 
In the following, we study each of the charts $\bar y=-1$ and $\bar \epsilon=1$ again. The results are summarized in \figref{blowup3}.
\begin{figure}[!t]
        \centering
        \includegraphics[width=0.65\textwidth]{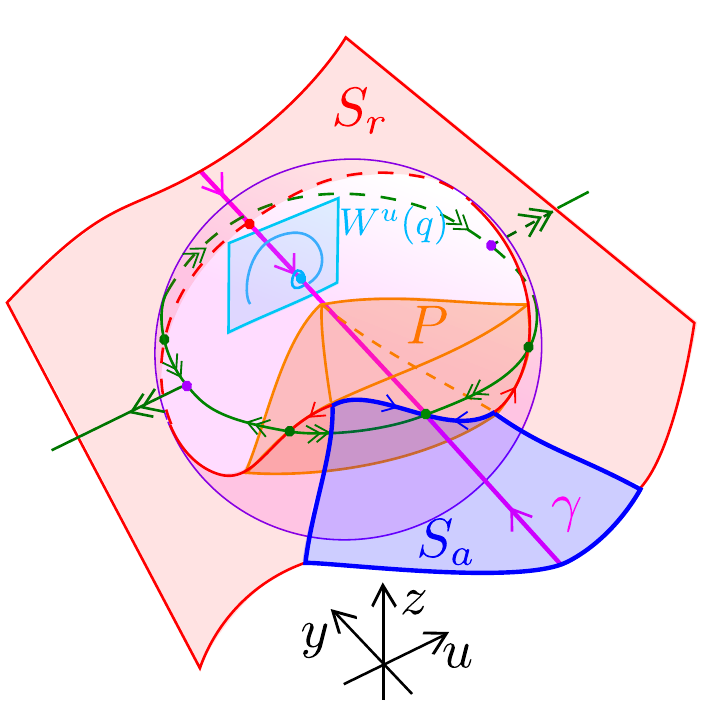}
        \caption{Illustration of the spherical blowup of the cusped saddle-node, using the same perspective as in \figref{blowup2}. In this case, we obtain a slow-fast system on the blowup sphere with $\gamma$ as a critical manifold. The reduced problem on $\gamma$ has an equilibrium $q$ which undergoes a Hopf bifurcation for the full system. In particular, on one side of the bifurcation $q$ is of saddle-focus type (the cyan surface illustrates the unstable manifold $W^s(q)$) and this is where an increased number of SAOs occur. The fast subsystem of the slow-fast system on the blowup sphere is of Lienard-type and this gives rise to a cylinder $P$ of limit cycles on the blowup sphere (in orange).}
        \figlab{blowup3}
\end{figure}
\subsection{Analysis in the $\bar y=-1$-chart}
The resulting equations can be obtained from \eqref{hatV1} upon substituting \eqref{c2expr}. We have
\begin{equation}\nonumber
\begin{aligned}
\dot r_1 &=-\frac{1}{2} r_1^3\epsilon_1 \left[-\sqrt{\epsilon_1} c_2+\left(-\frac{1}{2g}+\frac{3v_s}{2g}  u_1^2+ \response{\mathcal O(r_1^2)}\right)\right],\\
\dot u_1 &=-z_1-\frac{1}{g}\left(-3v_s +(9v_s^2+g)u_1^2+\mathcal O(r_1^2)\right)u_1\\
&+\frac{ 1}{2}r_1^2 u_1\epsilon_1 \left[-\sqrt{\epsilon_1} c_2+\left(-\frac{1}{2g}+\frac{3v_s}{2g}  u_1^2+ \response{\mathcal O(r_1^2)}\right)\right],\\
\dot z_1 &=\epsilon_1 \left(u_1 +\frac{3}{2}z_1r_1^2\left[-\sqrt{\epsilon_1}c_2+\left(-\frac{1}{2g}+\frac{3v_s}{2g}  u_1^2+ \response{\mathcal O(r_1^2)}\right)\right]\right),\\
\dot{\sqrt{\epsilon_1}} &=r_1^2 \epsilon_1\sqrt{\epsilon_1} \left[-\sqrt{\epsilon_1}c_2+\left(-\frac{1}{2g}+\frac{3v_s}{2g}  u_1^2+ \response{\mathcal O(r_1^2)}\right)\right],
\end{aligned}
\end{equation}
writing the last equation in terms of $\sqrt{\epsilon_1}$ rather than $\epsilon_1$ to indicate that the system is smooth in the former. For $r_1=\sqrt{\epsilon_1}=0$, we again find \eqref{z1u1} as a manifold of equilibria with the same stability properties. Therefore \propref{N1a} still applies, but the remainder is now smooth in $\sqrt{\epsilon_1}$. The reduced problem is then  
\begin{equation}\eqlab{redN1a2}
\begin{aligned}
 \dot r_1 &=-\frac12 r_1^3,\\
 \dot u_1 &=u_1 \left(-\frac{g^2}{3v_s}+\mathcal O(u_1^2,r_1^2,\sqrt{\epsilon_1})\right),\\
 \dot{\sqrt{\epsilon_1}} &=r_1^2\sqrt{\epsilon_1},
\end{aligned}
\end{equation}
after division of the right hand side by $\epsilon_1\left[-\frac{1}{2g}+\mathcal O(r_1,\sqrt{\epsilon_1})\right]$. Notice that the bracket is positive for all $r_1,\sqrt{\epsilon_1}\ge 0$ sufficiently small. From this we have.
\begin{proposition}
Fix any $c_2$ with $c$ as in \eqref{c2expr}, any $\delta>0$ sufficient small and consider
any point $p$ so that $\pi_a(p)\in S_a\backslash \gamma$. Then the following holds for all $0<\epsilon\ll 1$: The forward flow of $p$ intersects the section defined by $y=y_{\mathrm{in}}:=-(\epsilon\delta^{-1})^{1/2}$ in a point $(u,y,z)=(u_{\mathrm{in}},y_{\mathrm{in}},z_{\mathrm{in}})$ with 
\begin{align}
 u_{\mathrm{in}},z_{\mathrm{in}}=\mathcal O(e^{-\nu/\sqrt{\epsilon}}),
\end{align}
for some $\nu>0$. 
 \end{proposition}
\begin{proof}
 We work in the entry chart $\bar y=-1$, reduce to $N_{a,1}$, divide \eqref{redN1a2} by $\dot r_1$ and integrate from $r_1=\mathcal O(1)$ to $r_{1,\mathrm{in}}=\mathcal O(\epsilon^{1/4})$ (corresponding to the value of $y=y_{\mathrm{in}}$). \response{This leads to the estimate 
 \begin{align*}
  \vert u_{1,\mathrm{in}}\vert\le C e^{\nu \int_{r_{1}}^{r_{1,\mathrm{in}}} s^{-3} ds} = Ce^{\frac12 \nu r_{1}^{-2}} e^{-\frac12 \nu r_{1,\mathrm{in}}^{-2}}
 \end{align*}
for some $C>0$ and $\nu>0$ independent of $\epsilon$. }
\end{proof}
Next, we notice the following: Consider the $r_1=0$ subsystem:
\begin{equation}\eqlab{slowfastlienard}
\begin{aligned}
\dot u_1 &=-z_1-\frac{1}{g}\left(-3v_s +(9v_s^2+g)u_1^2\right)u_1,\\
\dot z_1 &=\epsilon_1 u_1,\\
\dot{\epsilon_1} &=0.
\end{aligned}
\end{equation}
This system is a slow-fast Lienard system in the $(u_1,z_1)$-plane with $\epsilon_1\ge 0$ as the small parameter. The analysis is straightforward and illustrated in \figref{lienard}. In particular, the associated layer problem has the set \eqref{z1u1} as a manifold of equilibria, being attracting for $u_1\in (-u_p,u_p)$ and repelling for $u_1\notin [-u_{p,1},u_{p,1}]$, recall \eqref{u1p}. The reduced problem has a stable node at $(u_1,z_1)=0$ on the attracting branch and we are therefore in the ``relaxation regime'', but the relaxation oscillations for $\epsilon_1>0$ small enough are repelling. \response{(Notice that in contrast to \eqref{fhnuncp}, the middle branch of the critical manifold of \eqref{slowfastlienard} is attracting. Compare also \figref{uncoupled} with \figref{lienard}.)}
Therefore we have the following:
\begin{lemma}
 On $r_1=0$ there exists an invariant cylinder $P_{1}$, contained within $\epsilon_1\in [0,\delta]$, for $\delta>0$ small enough, such that $P_1(\epsilon_{10}):=P_1\cap \{\epsilon_1=\epsilon_{10}\}$ is a repelling limit cycle for each $\epsilon_{10}\in (0,\delta]$. In particular, $P_1(0)$ is a singular slow-fast relaxation cycle. 
\end{lemma}
\begin{figure}[!t]
        \centering
        \includegraphics[width=0.65\textwidth]{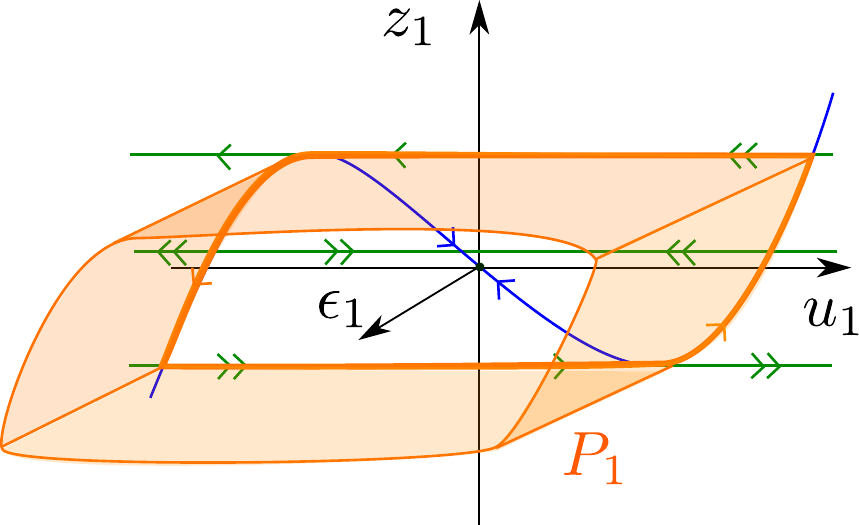}
        \caption{The invariant cylinder $P_1$ in the \response{$\bar y=-1$}-chart within $r_1=0$. For $\epsilon_1=0$, it becomes a singular van der Pol-like relaxation cycle in the $(u_1,z_1)$-plane. }
        \figlab{lienard}
\end{figure}

\subsection{Analysis in the $\bar \epsilon=1$-chart}
The resulting equations can be obtained from \eqref{hatV2} upon substituting \eqref{c2expr}. We have
\begin{equation}\eqlab{hatV21}
\begin{aligned}
 \dot u_2 &=-z_2-\frac{1}{g} \left(3v_s  y_2 +(9v_s^2+g) u_2^2 +\mathcal O(r_2^2)\right)u_2,\\
 \dot y_2 &= r_2^2\left(\response{-}c_2 +\frac{1}{2g}y_2+\frac{3v_s}{2g}  u_2^2+ \mathcal O(r_2^2)\right),\\
 \dot z_2 &=u_2,
\end{aligned}
\end{equation}
and $\dot r_2=0$. This is now a slow-fast system with two fast variables, $u_2$ and $z_2$, and one single slow variable $y_2$. In particular, we notice that 
\begin{align*}
 \gamma_2:\quad u_2=z_2=0,\,y_2\in \mathbb R,
\end{align*}
is now a critical manifold for $r_2=0$. In fact, the associated fast sub-system 
\begin{equation}\eqlab{layeru2z2}
\begin{aligned}
 \dot u_2 &=-z_2-\frac{1}{g} \left(3v_s  y_2 +(9v_s^2+g) u_2^2\right)u_2,\\
  \dot z_2 &=u_2,
\end{aligned}
\end{equation}
with $y_2$ fixed as a parameter for $r_2=0$, is a Lienard equation. 
\begin{lemma}\lemmalab{lienard}
 The system \eqref{layeru2z2} has a unique repelling limit cycle $P_2(y_2)$ for each $y_2<0$. 
\end{lemma}
\begin{proof}
This follows from Lienard's theorem \cite{perko2001a}. In fact, \eqref{layeru2z2} is topologically equivalent with the van-der Pol system in backward time. In particular, there is a subcritical Hopf bifurcation of \eqref{layeru2z2} at $y_2=0$. 
\end{proof}
By uniqueness, the set $P_2$ coincides with $P_1$ upon using the change of coordinates \eqref{cc1} where these overlap.

By Fenichel's theory \cite{fen3}, the manifold $P_2$ of repelling limit cycles of \eqref{hatV21} for $r_2=0$, perturbs as an invariant manifold $P_{2,r_2}$ within compact subsets. It creates a funnel region, \response{where trajectories inside contract towards $\gamma_2$, while trajectories outside get repelled away from the local neighborhood of the cusp}.  
On the perturbed cylinder, repelling limit cycles may exist. This depends upon $c_2$. Indeed, the \response{reduced problem} on $P_2$ is given by averaging: Let $T(y_2)$ be the period of $P_2(y_2)$ as a periodic orbit $(u_2(t;y_2),z_2(t;y_2))$ of \eqref{layeru2z2}. Then 
\begin{align}
 y_2' = \response{-}c_2 +\frac{1}{2g}y_2+\frac{3v_s}{2g} \frac{1}{T(y_2)} \int_0^{T(y_2)} u_2(t;y_2)^2 dt,\eqlab{redP2}
\end{align}
on $P_2$. Consequently, the reduced problem has an equilibrium at $y_2$ for the parameter value $c_2$ whenever
\begin{align}
 c_2 =\response{ \frac{1}{2g}y_{2}+\frac{3v_s}{2g} \frac{1}{T(y_{2})} \int_0^{T(y_{2})} u_2(t;y_{2})^2 dt},\quad y_{2}<0.\eqlab{c2eq}
\end{align}
Notice that $y_{2}=0$ on the right hand side gives $c_2=0$. It is \response{possible} to show that $u(t,y_2) = 2\sqrt{\frac{-y_2v_s}{9v_s^2+g}} \cos(t)+\mathcal O(y_2)$ (\response{using e.g. a Melnikov computation, see \cite{kristiansen2022}, where a similar computation is performed in a related context}). This gives a linear approximation of the right hand side of \eqref{c2eq}:
\begin{align}
 c_2 \approx \frac{3v_s^2+g}{2g(9v_s^2+g)}y_2.\eqlab{c2lin}
\end{align}
Consequently, the right hand side is a decreasing function of $y_2$ for $y_2<0$ small enough for $g<0$. Numerical computations (see \figref{c2vsy2P2}) indicate that this holds for all $y_2<0$. We have not found a way to show this, but if we assume this, then we have the following result.

\begin{figure}[!t]
        \centering
        \includegraphics[width=0.75\textwidth]{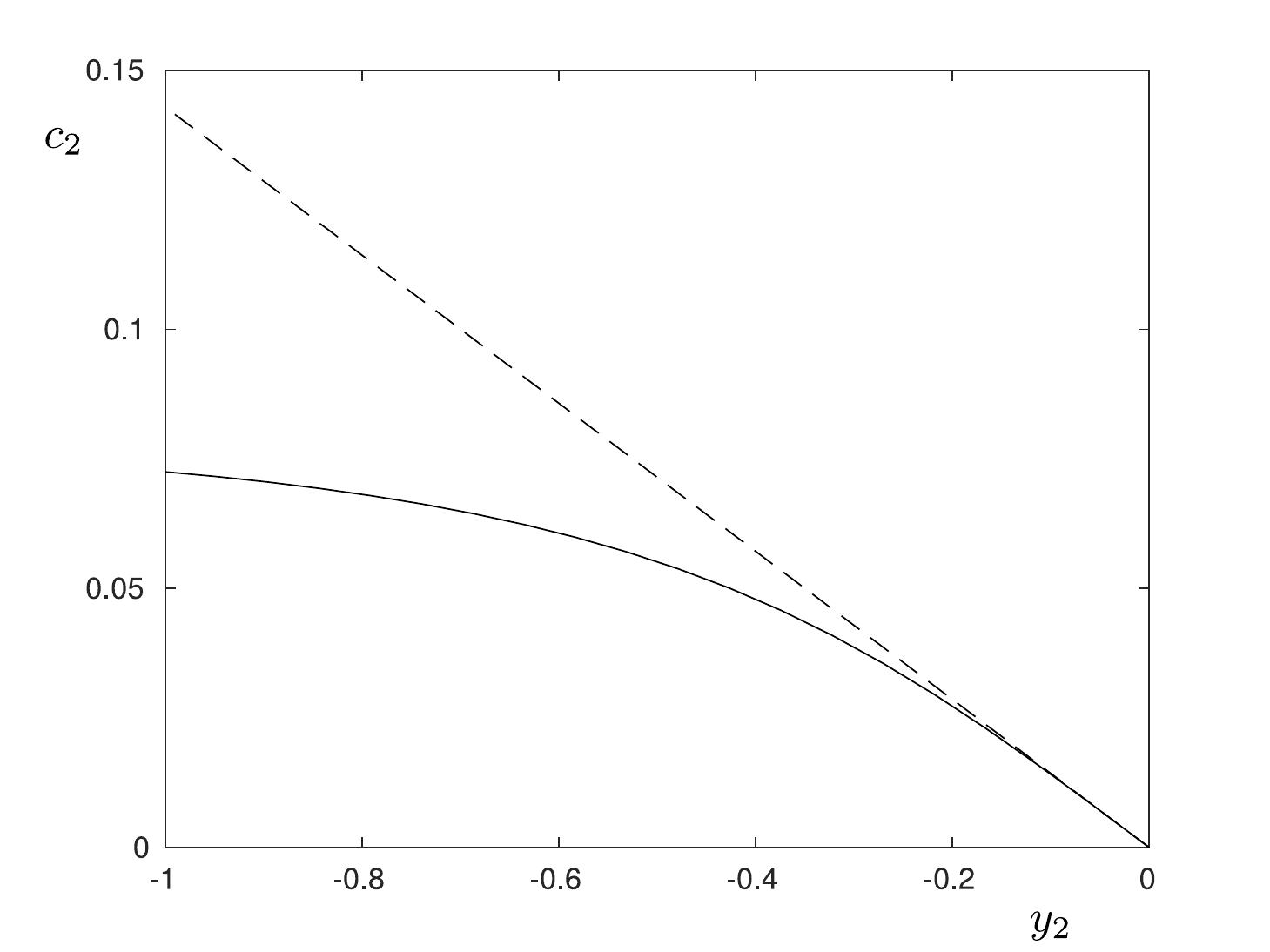}
        \caption{The right hand side of \eqref{c2eq} as a function of $y_2$ for $g=-1$. The dotted line is the linear approximation \eqref{c2lin} obtained through Melnikov. In order to compute the full line, we have first computed an accurate approximation of a limit cycle of \eqref{layeru2z2} (using shooting and Newton's method) and then subsequently computed the average. }
        \figlab{c2vsy2P2}
\end{figure}

\begin{proposition}\proplab{contractg2}
 Suppose that the right hand side of \eqref{c2eq} is a strictly \response{decreasing} function of $y_2<0$. Fix any $c_{20}>0$ and let $y_{20}$ be the unique value $y_2$ such that \eqref{c2eq} holds with $c_2=c_{20}$. Then the reduced problem \eqref{redP2} on $P_2$ has a unique attracting fixed point at $y_2=y_{20}$ for the parameter value $c_2=c_{20}$.
 
 Moreover, for all $0<r_2\ll 1$, the corresponding singular cycle $P_2(y_{20})$ then perturbs to a hyperbolic (saddle-type) limit cycle $P_{2,r_2}(y_{20})$ of \eqref{hatV21}  for $c_2=c_{20}$. This limit cycle is $\mathcal O(r_2^2)$-close to $P_{2}(y_{20})$. 
\end{proposition}
As $c_{2}$ ranges over a compact subset $I$ of $(0,\infty)$, we then obtain a family of repelling limit cycles on $P_{2,r_2}$ for all $0<\epsilon\le \epsilon_0(I)$. Recall that $r_2=\epsilon^{1/4}$. It is possible to show that the family $P_{2,r_2}$ overlaps with the repelling Hopf cycles emanating from $y_2=0$ at $c_2=0$, recall \remref{lyapunov}.

\begin{remark}\remlab{lyapunov}
 \response{The Liapunov coefficient $l_1$ (recall \eqref{lyapunov}) of the Hopf bifurcation for $c_2=0$ (corresponding to $c=v_s(g)$) can be calculated from \eqref{hatV21}. Indeed, a direction calculation shows that the two-dimensional center manifold at the Hopf-point takes the following form 
 \begin{align*}
  y_2 \approx -\frac{3v_s}{2} u_2^2 -\frac{3v_s}{2} z_2^2,
 \end{align*}
 for $c_2,r_2\rightarrow 0$,
  up to and including quadratic order in $(u_2,z_2)$. On this center manifold, with $c_2,r_2\rightarrow 0$, we  
then have that
\begin{equation}\eqlab{hopflin}
\begin{aligned}
 \dot u_2 &=-z_2 +f(u_2,z_2),\\
 \dot z_2 &=u_2,
\end{aligned}
\end{equation}
with 
\begin{align*}
 f(u_2,z_2) \approx -\frac{1}{g} \left(-\frac92 v_s^2 z_2^2 +\left(\frac{9}{2}v_s^2+g\right) u_2^2 \right)u_2,
\end{align*}
up to an including cubic order in $(u_2,z_2)$.
The system \eqref{hopflin} is already in normal form and we therefore have that 
\begin{align*}
 \hat l_1:=\frac{1}{16} \left(\frac{\partial^3 f}{\partial u_2^3}(0,0)+\frac{\partial^3 f}{\partial u_2\partial z_2^2}(0,0)\right)=\frac{3(g-3)}{8g}
\end{align*}
using \cite[Equation 3.4.11]{guckenheimer97}. This gives the leading order expression in \eqref{lyapunov} upon division by $r_2^2=\sqrt{\epsilon}$. This division corresponds to the desingularization in the chart $\bar \epsilon=1$, recall \eqref{hatV}. }
\end{remark}


We now proceed to study the properties of the critical manifold $$\gamma_2:\quad u_2=z_2=0,y_2\in \mathbb R,$$ of \eqref{hatV21} for $r_2=0$. The linearization of \eqref{layeru2z2} around $u_2=z_2=0$ gives
\begin{align}
 \begin{pmatrix}
   -\frac{3v_s}{g} y_2 & -1\\
   1 & 0
 \end{pmatrix}.\eqlab{matrixJac}
\end{align}
\response{The eigenvalues are imaginary $\pm i$ for $y_2=0$ due to the Hopf.}
From this we can easily deduce the stability properties. 
\begin{lemma}\lemmalab{L2stability}
 The critical manifold $\gamma_2$ of \eqref{hatV21} for $r_2=0$ is normally hyperbolic for $y_2\ne 0$. The subset $\gamma_2^a$ with $y_2<0$ is attracting whereas the subset $\gamma_2^r$ with $y_2>0$ is repelling. 
 Moreover, $\gamma_2^r=\gamma_2^{rf}\cup \gamma_2^{rn}$ where  $\gamma_2^{rf}$ is the subset of $\gamma_2$ with $y_2\in \left(0,-\frac{2g}{3v_s}\right)$ having normal focus stability (i.e. the eigenvalues of \eqref{matrixJac} are complex conjugated with positive real part) whereas $\gamma_2^{rn}$ is the subset of $\gamma_2^r$ with $y_2\ge -\frac{2g}{3v_s}$ having normal nodal stability (i.e. the eigenvalues of \eqref{matrixJac} are real and positive). 
\end{lemma}
There is a similar division of $\gamma_2^a=\gamma_2^{af}\cup \gamma_2^{an}$ for $y_2\in \left(\frac{2g}{3v_s},0\right)$ and $y_2\le \frac{2g}{3v_s}$, respectively, but this will be less important. 

The reduced problem on $\gamma_2$ is given by
\begin{align*}
 y_2' &=-c_2 + \frac{1}{2g}y_2.
\end{align*}
It has a hyperbolic and attracting equilibrium at $y_2= 2c_2 g$. In combination with \lemmaref{L2stability}, we realize the following.
\begin{lemma}\lemmalab{q2c2}
Let $q_2$ denote the equilibrium $(u_2,y_2,z_2)=(0,2c_2g,0)$ which is hyperbolic and attracting for the reduced problem on $\gamma_2$. Then the following holds.
 \begin{itemize}
  \item For $c_2 > 0$, then $q_2$ sits on the attracting part of $\gamma_2$ and it perturbs to an attracting equilibrium \eqref{hatV21} for all $0<\epsilon\ll 1$.
  \item For $c_2\in \left(-\frac{1}{3v_s},0\right)$, then $q_2\in \gamma_2^{rf}$ and it perturbs to a saddle-focus equilibrium of \eqref{hatV21} for all $0<\epsilon\ll 1$ with a one-dimensional stable manifold along $\gamma_2$ and a two-dimensional unstable manifold with focus-type dynamics. 
  \item For $c_2< -\frac{1}{3v_s}$, then $q_2\in \gamma_2^{rn}$ and it perturbs to a saddle equilibrium of \eqref{hatV21} for all $0<\epsilon\ll 1$ with a one-dimensional stable manifold along $\gamma_2$ and a two-dimensional unstable manifold with nodal-type dynamics. 
 \end{itemize}

\end{lemma}

We illustrate the findings in the $\bar \epsilon=1$-chart in \figref{cuspednode2}. See figure caption for further details. We are now ready to describe our main result on small amplitude oscillations for the cusped saddle-node. 
\begin{theorem}\thmlab{main2}
\response{Consider $c$ as in \eqref{c2expr} with $c_2\in \left(-\frac{1}{3v_s},0\right)$ fixed and any point $p$ so that $\pi_a(p)\in S_a\backslash \gamma$. Then the following holds for all $0<\epsilon\ll 1$: The forward orbit of $p$ intersects the section defined by $y=0$ in a point $(u,y,z)=(u_0,0,z_0)$ with 
\begin{align}\eqlab{amplitudeuzexp}
 u_0,z_0=\mathcal O(e^{-c/\sqrt{\epsilon}}).
\end{align}
The number of SAOs of the forward orbit is unbounded as $\epsilon\rightarrow 0$, but finitely many are $\mathcal O(1)$ in amplitude in the $(u_2,z_2)$-plane. }
  \end{theorem}
  \begin{proof}
  \response{
  For $c_2\in \left(-\frac{1}{3v_s},0\right)$, $q_2$ belongs to $\gamma_2^{rf}$ and is of saddle-focus type, recall \lemmaref{q2c2}. 
  \eqref{amplitudeuzexp} follows directly from the exponentially contraction $e^{- c \tau/\sqrt{\epsilon}}$ towards the invariant $\gamma_2^a$ on the slow time scale $\tau$ of \eqref{hatV21}; recall that $r_2=\epsilon^{1/4}$. Due to the focus behavior of $\gamma_2$ near $y_2=0$, recall \lemmaref{L2stability}, the forward orbit will experience an unbounded number of SAOs as $\epsilon\rightarrow 0$. These will be exponentially small in amplitude. Moreover, since $\pi_a(p)\in S_a\backslash \gamma$ and $\gamma_2$ is the stable manifold of $q_2$, the forward orbit of $p$ will extend along $\gamma_2^{rf}$, remaining exponential close for all $y_2\in [0,y_{21}(c_2)]$, for some $0<y_{21}(c_2)<2c_2g$. Beyond this, the orbit will eventually be repelled away from $\gamma_2$ due to the unstable manifold of $q_2$. Since  $q_2\in \gamma_2^{rf}$ for $c_2\in \left(-\frac{1}{3v_s},0\right)$, we obtain finitely many $O(1)$ SAOs due to the focus dynamics in the $(u_2,z_2)$-projection at some distance from $q_2$. This completes the proof.
%
  } 
%
  \end{proof}
  \begin{remark}\remlab{remfinal}
    \response{With the assumptions of \thmref{main2}, there is a bifurcation delay along $\gamma_2$. For the statement of the theorem, we did not need to determine this delay in details. However, due to the invariance of $\gamma_2$, it can be determined by a way-in/way-out function in the following way:
Let 
\begin{align}\eqlab{nupm}
\nu_{\pm}(y_2)=-\frac{3v_s}{2g}y_2\pm \frac12 \sqrt{\frac{9v_s^2}{4g^2}y_2^2-4},
\end{align} denote the eigenvalues of \eqref{matrixJac}. Then for $c_2<0$ the exit point  $y_{2,\mathrm{exit}}\in (0,2c_2g)$ is for $r_2\rightarrow 0$ determined by
\begin{align}
 \int_{-\infty}^{y_{2,\mathrm{exit}}} \frac{\operatorname{Re}\nu_+(y_2)}{-c_2+\frac{1}{2g}y_2} dy_2 = 0.\eqlab{entryexit}
\end{align}
(The integral is convergent since $\frac{\operatorname{Re}\nu_+(y_2)}{-c_2+\frac{1}{2g}y_2}\approx \frac{-4g^2}{3v_s y_2^2}$ for $y_2\rightarrow -\infty$ and $y_{2,\mathrm{exit}}>0$ exists and is unique for each $c_2<0$ since $\frac{\operatorname{Re}\nu_+(y_2)}{-c_2+\frac{1}{2g}y_2}\rightarrow \infty$ for $y_2\rightarrow 2c_2 g^-$. $y_{2,\mathrm{exit}}(c_2)$ is also continuous and $y_{2,\mathrm{exit}}(0^-)=0$.) The integral gives a lengthy expression and we have not found a way to solve for $y_{2,\mathrm{exit}}$. We therefore only present a diagram (obtained in Matlab) for $g=-1$, see \figref{entryexit} and the figure caption for further details, of $y_{2,\mathrm{exit}}$ as a function of $c_2$. 
%
%
%
%
%
%
%
%
%
%
%

Due the invariance of $\gamma_2$, the delay for our system \eqref{fhn} is different from the bifurcation delay for the folded saddle-node, see e.g.~\cite{krupa2010a}. Indeed, for the folded saddle-node, the delay for \textit{analytic systems} depends upon (following \cite{neishtadt1987a,neishtadt1988a}) buffer points. If we were to break the symmetry of \eqref{fhn}, then one would like to rely on the same methods. But this could be problematic in this context, since the center manifold in \propref{cmred} is not expected to be analytic. }
\end{remark}
  \begin{figure}[!t]
        \centering
        \includegraphics[width=0.7\textwidth]{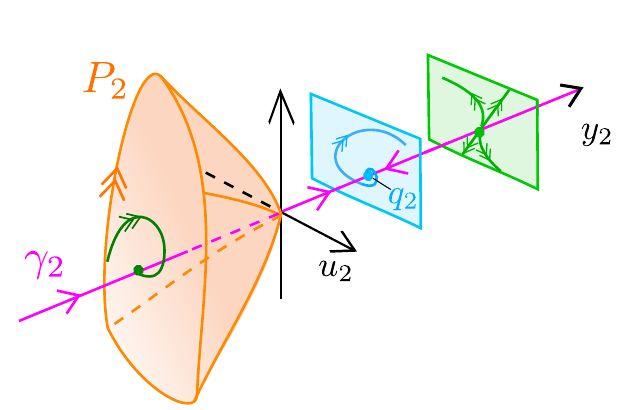}
        \caption{Illustration of the dynamics in the $\bar \epsilon=1$-chart in the case of the cusped saddle-node. The manifold of limit cycles $P_2$ is in orange while the critical manifold $\gamma_2$ is in pink. On the positive side of $y_2=0$, we illustrate the normal dynamics on $\gamma_2^{rf}$ in cyan (focus type) and on $\gamma_2^{rf}$ in green (nodal type). When the equilibrium $q_2\in \gamma_2$ (also cyan) lies on $\gamma_2^{rf}$,  SAOs of order $\mathcal O(1)$ (in the $(u_2,y_2,z_2)$-scaling) occur near $W^u(q_2)$.}
        \figlab{cuspednode2}
\end{figure}

  \begin{figure}[!t]
        \centering
        \includegraphics[width=0.7\textwidth]{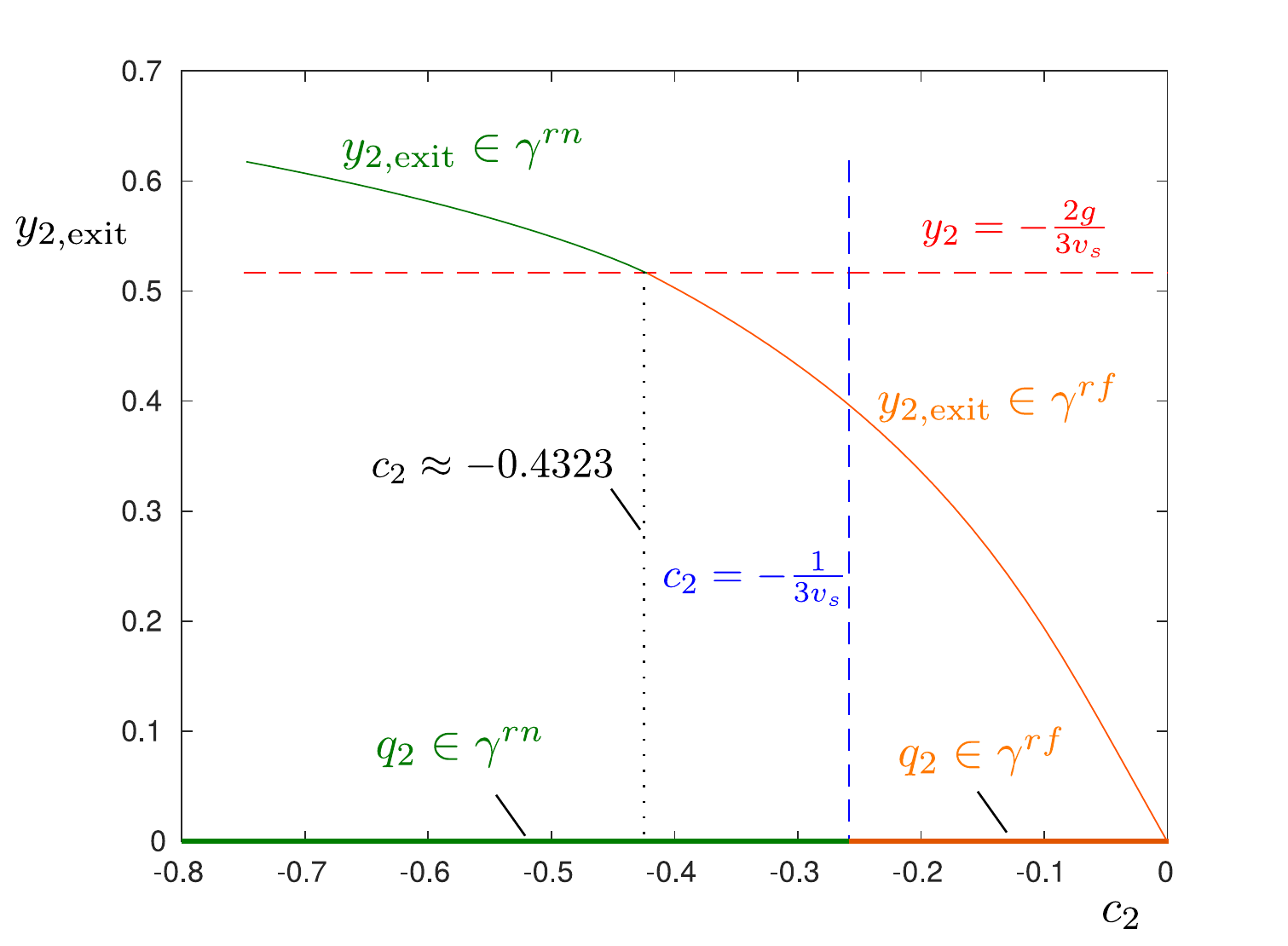}
        \caption{\response{The graph of $y_{2,\mathrm{exit}}(c_2)$ obtained from the equation \eqref{entryexit} with $g=-1$ (in orange and in green). The blue line is $c_2=-\frac{1}{3v_s}$, i.e. the value of $c_2$ such that $q_2$ is an improper node of the fast sub-system \eqref{layeru2z2} (i.e. $\nu_-=\nu_+$, see \eqref{nupm}). The red line is the corresponding $y_2$-value of $q_2$: $y_2=-\frac{2g}{3v_s}$. There is an intersection of the graph of $y_{2,\mathrm{exit}}(c_2)$ with $y_2=-\frac{2g}{3v_s}$ at $c_2\approx -0.42$. This intersection divides the graph into two parts, indicated in orange and green where $(0,y_{2,\mathrm{exit}},0)\in \gamma_2^{rf}$ and $(0,y_{2,\mathrm{exit}},0)\in \gamma_2^{rn}$, respectively. (In the figure, we have abused notation slightly and written this more compactly as $y_{2,\mathrm{exit}}\in \gamma_2^{rf}$ and $y_{2,\mathrm{exit}}\in \gamma_2^{rn}$.) }}
        \figlab{entryexit}
\end{figure}

 A similar result holds for $c_2< -\frac{1}{3v_s}$, but 
 due to the normal nodal dynamics along $\gamma^{rn}$ all SAOs may be exponentially small in this case. \response{We see this in \figref{entryexit} for the value of $g=-1$. In particular, for $c_2<-0.42$ the exit point (green part of curve) is in the normal nodal regime where there are no additional $\mathcal O(1)$-oscillations when the trajectory separate from $\gamma_2$.}
 For $c_2>0$, on the other hand, the forward flow of $p$ is attracted to the stable equilibrium near $q_2$. 
 
 \begin{remark}\remlab{cblowup}
Formally, the scaling \eqref{c2expr} does not overlap with the regime covered by \thmref{main1} where $c$ is fixed in a compact subset of $c<v_s$. There is therefore a gap that we do not cover in this paper. However, to cover this gap, and obtain a complete description of $c$ in a full neighborhood of $v_s$, one could include $c$ in the blowup transformation \eqref{blowup1} as follows 
\begin{align*}
 c = v_s+r^2 \bar c,
\end{align*}
and consider $(\bar u,\bar y,\bar z,\bar \epsilon,\bar c)\in S^4$. In particular, in this way, one could cover the gap by working in the directional chart corresponding to $\bar c=-1$. \response{Notice that the associated scaling chart $\bar \epsilon=1$ gives rise to the same coordinates $(u_2,y_2,z_2,r_2,c_2)$ where $c=v_s+r^2 c_2$ in agreement with \eqref{c2expr}. (This also motivates the use of the subscript on $c$, recall the convention before \remref{broercusp}.) } We shall not pursue this further in the present paper. 
\end{remark}

 \section{Conclusions}\label{sec:final}
\response{
In this paper, we have analyzed cusped singularities (cusped node and cusped saddle-node) and demonstrated that they form a mechanism for SAOs in two coupled FitzHugh-Nagumo units with symmetric and repulsive coupling. As for the folded node, we showed that the number of SAOs is determined by the Weber equation and the ratio of eigenvalues of the cusped node (upon desingularization). Similarly, we showed that the cusped saddle-node marks the onset of SAOs. Although there are many similarities between the folded singularities and the cusped versions studied in the present paper, there are also several differences, see e.g. \remref{strong} and \lemmaref{lienard}. Perhaps most importantly, our cusped node does not have a strong canard and there are also two fast directions away from the cusp ($u$ increasing and $u$ decreasing in \figref{lemma1}), as opposed to just one in the case of the standard folded singularity. 
The latter property also has consequences on MMOs and the LAOs that we see in \figref{MMO}. For the folded node, MMOs occur if there is return to the funnel region, see \cite{brons06}. The same is true in the present case, but it is slightly more subtle. Suppose (for definiteness) that there is a return mechanism to $S_a\backslash\gamma$, leaving the cusp region  along the positive $u$-direction. 
 Then as a consequence of \thmref{main1}, we obtain the following: 
{Let $\lfloor \frac{\lambda_2}{\lambda_1}\rfloor$ be even (odd) and suppose that the return to $S_a$ is on the $u$-positive side ($u$-negative side, respectively) of $\gamma$. Then we have (``one-sided'') MMOs for all $0<\epsilon\ll 1$ with $u$ always increasing upon passage through the cusp.} 
However, if affirmative, then the system \eqref{fhn} -- due to the symmetry $\mathcal S$ -- also has MMOs with $u$ always decreasing upon passage through the cusp. In fact, more generally, once we have a return to $S_a\backslash \gamma$ along one direction ($u$-positive or $u$-negative), then the symmetry give rise to a return along the other direction ($u$-negative or $u$-positive, respectively) too. We can then also have (``mixed'') MMOs where $u$ alternates sign upon passing through the cusp $f_1$. We see this 
in \figref{MMO} 
for $c=1.27$. Indeed, here there is an alternation between $v_1$ and $v_2$ being increasing ($v_2$, respectively, $v_1$ decreasing) which precisely corresponds to a change in sign in $u$. 
The description of the return mechanism for \eqref{fhn}, and whether we have ``one-sided'' or ``mixed'' MMOs, require a careful analysis of the layer problem \eqref{layer} but also of the reduced problem \eqref{reduced} (away from the cusp). We leave such an analysis to future work. 
}
 In future work, it would also be interesting to study the cusped singularities in a general setting without a symmetry. We already have some partial results in this direction. \response{The cusped node then becomes a co-dimension one bifurcation of a folded node that transverses the cusp upon parameter variation. In line with our findings, the number of SAOs does not change upon this passage. Within this context, it would also be interesting in future work to study the secondary canards and the role of a strong canard. }
 
 Similarly, the cusped saddle-node becomes co-dimension two without the symmetry. However, going from a folded saddle-node to a cusped saddle-node seems slightly more involved. A folded saddle-node (type II) is accompanied by a canard-like explosion of limit cycles (due to the strong canard), see also \cite{kristiansen2022}. In our symmetric cusped saddle-node there is no explosion, but instead a cylinder on which limit cycles occur, recall \propref{contractg2}. It is unclear how this scenario unfolds without the symmetry and how it precisely connects to the folded saddle-node. Moreover, a folded saddle-node actually comes in two versions. We have only focused on type II in this manuscript \cite{krupa_extending_2001}, but there is also a type I \cite{vo15}. Future research should also uncover how the generalized cusped saddle-node relates to these. 

\subsection*{Acknowledgment}
The authors are thankful for the discussions they have had with Morten Br{\o}ns in preparation of this manuscript. 
%

\begin{thebibliography}{10}
 \bibitem{arnold1984a}
{\sc V.~I. Arnold}, {\em Catastrophe Theory}, Springer Berlin Heidelberg,
  1984.
 
\bibitem{bar85}
{\sc K.~Bar-Eli}, {\em On the stability of coupled chemical oscillators},
  Physica D, 14 (1985), pp.~242--252.

\bibitem{battaglin21}
{\sc S.~Battaglin and M.~G. Pedersen}, {\em Geometric analysis of mixed-mode
  oscillations in a model of electrical activity in human beta-cells},
  Nonlinear Dynamics, 104 (2021), pp.~4445--4457.
  
  \bibitem{Belitskii1973}
{\sc G.~R.~Belitskii}, {\em Functional equations and conjugacy of local diffeomorphisms of a finite smoothness class}, Functional Analysis and Its Applications, 4 (1973), pp.~268--277,
  \url{https://doi.org/10.2307/2374346}.
  
 
\bibitem{broer2013a}
{\sc H.~W. Broer, T.~J. Kaper, and M.~Krupa}, {\em Geometric desingularization
  of a cusp singularity in slow-fast systems with applications to {Zeeman's}
  examples}, Journal of Dynamics and Differential Equations, 25 (2013),
  pp.~925--958, \url{https://doi.org/10.1007/s10884-013-9322-5}.

\bibitem{brons06}
{\sc M.~Br{\o}ns, M.~Krupa, and M.~Wechselberger}, {\em Mixed mode oscillations
  due to the generalized canard phenomenon}, Fields Inst. Commun., 49 (2006),
  pp.~39--63.

\bibitem{car1}
{\sc J.~Carr}, {\em Applications of centre manifold theory}, vol.~35, New York:
  Springer-Verlag, 1981.

\bibitem{curtu10}
{\sc R.~Curtu}, {\em Singular Hopf bifurcations and mixed-mode oscillations in a two-cell inhibitory neural network}, 
Physica D: Nonlinear Phenomena, 239 (2010),
  pp.~504--514.

\bibitem{curtu11}
{\sc R.~Curtu, and J.~Rubin}, 
{\em Interaction of canard and singular Hopf mechanisms in a neural model}, 
SIAM Journal on Applied Dynamical Systems, 10 (2011),
  pp.~1443--1479.
  


\bibitem{devries00}
{\sc G.~De~Vries and A.~Sherman}, {\em Channel sharing in pancreatic beta-cells
  revisited: enhancement of emergent bursting by noise}, J Theor Biol, 207
  (2000), pp.~513--30, \url{https://doi.org/10.1006/jtbi.2000.2193}.

\bibitem{desroches12}
{\sc M.~Desroches, J.~Guckenheimer, B.~Krauskopf, C.~Kuehn, H.~M. Osinga, and
  M.~Wechselberger}, {\em Mixed-mode oscillations with multiple time scales},
  SIAM Review, 54 (2012), pp.~211--288,
  \url{https://doi.org/10.1137/100791233}.

\bibitem{dickson00}
{\sc C.~T. Dickson, J.~Magistretti, M.~H. Shalinsky, E.~Frans{\'e}n, M.~E.
  Hasselmo, and A.~Alonso}, {\em Properties and role of {I(h)} in the pacing of
  subthreshold oscillations in entorhinal cortex layer {II} neurons}, J
  Neurophysiol, 83 (2000), pp.~2562--79,
  \url{https://doi.org/10.1152/jn.2000.83.5.2562}.

\bibitem{dumortier1996a}
{\sc F.~Dumortier and R.~Roussarie}, {\em Canard cycles and center manifolds},
  Memoirs of the American Mathematical Society, 121 (1996), pp.~1--96.

\bibitem{ermentrout90}
{\sc G.~Ermentrout and N.~Kopell}, {\em Oscillator death in systems of coupled
  neural oscillators}, SIAM J Appl Math, 50 (1990), pp.~125--146.

\bibitem{fen3}
{\sc N.~Fenichel}, {\em Geometric singular perturbation theory for ordinary
  differential equations}, J. Diff. Eq., 31 (1979), pp.~53--98.

\bibitem{fitzhugh61}
{\sc R.~Fitzhugh}, {\em Impulses and physiological states in theoretical models
  of nerve membrane}, Biophys J, 1 (1961), pp.~445--466,
  \url{https://doi.org/10.1016/s0006-3495(61)86902-6}.

\bibitem{gregor10}
{\sc T.~Gregor, K.~Fujimoto, N.~Masaki, and S.~Sawai}, {\em The onset of
  collective behavior in social amoebae}, Science, 328 (2010), pp.~1021--5,
  \url{https://doi.org/10.1126/science.1183415}.

\bibitem{guckenheimer08}
{\sc J.~Guckenheimer}, {\em Singular {Hopf} bifurcation in systems with two
  slow variables}, SIAM Journal on Applied Dynamical Systems, 7 (2008),
  pp.~1355--1377.

 \bibitem{guckenheimer97},
	{\sc J.~Guckenheimer and P.~Holmes},
	{\em Nonlinear Oscillations, Dynamical Systems and Bifurcations of Vector Fields},
	Springer,
1997.
	
\bibitem{gutfreund95}
{\sc Y.~Gutfreund, Y.~Yarom, and I.~Segev}, {\em Subthreshold oscillations and
  resonant frequency in guinea-pig cortical neurons: physiology and modelling},
  J Physiol, 483 ( Pt 3) (1995), pp.~621--40,
  \url{https://doi.org/10.1113/jphysiol.1995.sp020611}.

\bibitem{haragus2011a}
{\sc M.~Haragus and G.~Iooss}, {\em Local Bifurcations, Center Manifolds, and
  Normal Forms in Infinite-Dimensional Dynamical Systems}, EDP Sciences, 2011.

\bibitem{hodgkin52}
{\sc A.~L. Hodgkin and A.~F. Huxley}, {\em A quantitative description of
  membrane current and its application to conduction and excitation in nerve},
  J Physiol, 117 (1952), pp.~500--44.

\bibitem{izhikevich07}
{\sc E.~M. Izhikevich}, {\em Dynamical systems in neuroscience}, MIT press,
  2007.

\bibitem{jard2016a}
{\sc H.~Jard\'on-Kojakhmetov, H.~W. Broer, and R.~Roussarie}, {\em Analysis of
  a slow-fast system near a cusp singularity}, Journal of Differential
  Equations, 260 (2016), pp.~3785--3843,
  \url{https://doi.org/10.1016/j.jde.2015.10.045}.

\bibitem{jones_1995}
{\sc C.~Jones}, {\em Geometric Singular Perturbation Theory, Lecture Notes in
  Mathematics, Dynamical Systems (Montecatini Terme)}, Springer, Berlin, 1995.

\bibitem{kimrey20}
{\sc J.~Kimrey, T.~Vo, and R.~Bertram}, {\em Big ducks in the heart: Canard
  analysis can explain large early afterdepolarizations in cardiomyocytes},
  SIAM Journal on Applied Dynamical Systems, 19 (2020), pp.~1701--1735.

\bibitem{kristiansen2020a}
{\sc K.~U. Kristiansen}, {\em On the pitchfork bifurcation of the folded node
  and other unbounded time-reversible connection problems in $\mathbb{R}^3$},
  {SIAM} Journal on Applied Dynamical Systems, 19 (2020), pp.~2059--2102,
  \url{https://doi.org/10.1137/20M1326180}.
\bibitem{kristiansen2022}
{\sc K.~U. Kristiansen}, {\em The dud canard: Existence of strong canard cycles in $\mathbb R^3$},
  arXiv:2207.00875v2 preprint, (2022).   	
  
\bibitem{krupa14}
{\sc M.~Krupa, B.~Ambrosio, and M.~Aziz-Alaoui}, {\em Weakly coupled
  two-slow--two-fast systems, folded singularities and mixed mode
  oscillations}, Nonlinearity, 27 (2014), p.~1555.

\bibitem{krupa_extending_2001}
{\sc M.~Krupa and P.~Szmolyan}, {\em Extending geometric singular perturbation
  theory to nonhyperbolic points - fold and canard points in two dimensions},
  {SIAM} Journal on Mathematical Analysis, 33 (2001), pp.~286--314,
  \url{http://epubs.siam.org/doi/abs/10.1137/S0036141099360919} (accessed
  2014-06-02).

\bibitem{krupa2010a}
{\sc M.~Krupa and M.~Wechselberger}, {\em Local analysis near a folded
  saddle-node singularity}, Journal of Differential Equations, 248 (2010),
  pp.~2841--2888, \url{https://doi.org/10.1016/j.jde.2010.02.006}.

\bibitem{laing02}
{\sc C.~R.~Laing, and C.~Carson},
{\em A spiking neuron model for binocular rivalry}, 
Journal of Computational Neuroscience, 12 (2002), pp.~39--53.

\bibitem{loppini15}
{\sc A.~Loppini, M.~Braun, S.~Filippi, and M.~G. Pedersen}, {\em Mathematical
  modeling of gap junction coupling and electrical activity in human
  $\beta$-cells}, Phys Biol, 12 (2015), p.~066002,
  \url{https://doi.org/10.1088/1478-3975/12/6/066002}.

\bibitem{loppini18}
{\sc A.~Loppini and M.~G. Pedersen}, {\em Gap-junction coupling can prolong
  beta-cell burst period by an order of magnitude via phantom bursting}, Chaos,
  28 (2018), p.~063111, \url{https://doi.org/10.1063/1.5022217}.

\bibitem{nagumo62}
{\sc J.~Nagumo, S.~Arimoto, and S.~Yoshizawa}, {\em An active pulse
  transmission line simulating nerve axon}, Proceedings of the IRE, 50 (1962),
  pp.~2061--2070.

\bibitem{neishtadt1987a}
{\sc A.~Neishtadt}, {\em Persistence of stability loss for dynamical
  bifurcations .1}, Differential Equations, 23 (1987), pp.~1385--1391.

\bibitem{neishtadt1988a}
{\sc A.~Neishtadt}, {\em Persistence of stability loss for dynamical
  bifurcations .2}, Differential Equations, 24 (1988), pp.~171--176.

\bibitem{pedersen05b}
{\sc M.~G. Pedersen}, {\em A comment on noise enhanced bursting in pancreatic
  beta-cells}, J Theor Biol, 235 (2005), pp.~1--3,
  \url{https://doi.org/10.1016/j.jtbi.2005.01.025}.

\bibitem{pedersen22}
{\sc M.~G. Pedersen, M.~Br{\o}ns, and M.~P. S{\o}rensen}, {\em
  Amplitude-modulated spiking as a novel route to bursting: Coupling-induced
  mixed-mode oscillations by symmetry breaking}, Chaos, 32 (2022), p.~013121,
  \url{https://doi.org/10.1063/5.0072497}.

\bibitem{perko2001a}
{\sc L.~Perko}, {\em Differential equations and dynamical systems}, Springer,
  2001.

\bibitem{riz14}
{\sc M.~Riz, M.~Braun, and M.~G. Pedersen}, {\em Mathematical modeling of
  heterogeneous electrophysiological responses in human $\beta$-cells}, PLoS
  Comput Biol, 10 (2014), p.~e1003389,
  \url{https://doi.org/10.1371/journal.pcbi.1003389}.

\bibitem{rotstein06}
{\sc H.~G. Rotstein, T.~Oppermann, J.~A. White, and N.~Kopell}, {\em The
  dynamic structure underlying subthreshold oscillatory activity and the onset
  of spikes in a model of medial entorhinal cortex stellate cells}, J Comput
  Neurosci, 21 (2006), pp.~271--92,
  \url{https://doi.org/10.1007/s10827-006-8096-8}.

\bibitem{rubin07}
{\sc J.~Rubin and M.~Wechselberger}, {\em Giant squid-hidden canard: the {3D}
  geometry of the {Hodgkin}-{Huxley} model}, Biol Cybern, 97 (2007), pp.~5--32,
  \url{https://doi.org/10.1007/s00422-007-0153-5}.

\bibitem{sherman94}
{\sc A.~Sherman}, {\em Anti-phase, asymmetric and aperiodic oscillations in
  excitable cells--{I}. {Coupled} bursters}, Bull Math Biol, 56 (1994),
  pp.~811--35.

\bibitem{sherman92}
{\sc A.~Sherman and J.~Rinzel}, {\em Rhythmogenic effects of weak electrotonic
  coupling in neuronal models}, Proc Natl Acad Sci U S A, 89 (1992),
  pp.~2471--4.

\bibitem{szmolyan_canards_2001}
{\sc P.~Szmolyan and M.~Wechselberger}, {\em Canards in $\mathbb{R}^3$}, J.
  Diff. Eq., 177 (2001), pp.~419--453,
  \url{https://doi.org/10.1006/jdeq.2001.4001},
  \url{http://linkinghub.elsevier.com/retrieve/pii/S002203960194001X} (accessed
  2014-05-26).

\bibitem{szmolyan2004a}
{\sc P.~Szmolyan and M.~Wechselberger}, {\em Relaxation oscillation in
  $\mathbb{R}^3$}, Journal of Differential Equations, 200 (2004), pp.~69--104,
  \url{https://doi.org/10.1016/j.jde.2003.09.010}.

\bibitem{tabak07}
{\sc J.~Tabak, N.~Toporikova, M.~E. Freeman, and R.~Bertram}, {\em Low dose of
  dopamine may stimulate prolactin secretion by increasing fast potassium
  currents}, J Comput Neurosci, 22 (2007), pp.~211--22,
  \url{https://doi.org/10.1007/s10827-006-0008-4}.

\bibitem{vo10}
{\sc T.~Vo, R.~Bertram, J.~Tabak, and M.~Wechselberger}, {\em Mixed mode
  oscillations as a mechanism for pseudo-plateau bursting}, J Comput Neurosci,
  28 (2010), pp.~443--58, \url{https://doi.org/10.1007/s10827-010-0226-7}.

\bibitem{vo15}
{\sc T.~Vo and M.~Wechselberger}, {\em Canards of folded saddle-node type {I}},
  SIAM Journal on Mathematical Analysis, 47 (2015), pp.~3235--3283.

\bibitem{wang92}
{\sc X.-J.~Wang, and J.~Rinzel},
{\em Alternating and synchronous rhythms in reciprocally inhibitory model neurons},
Neural Computation, 4 (1992), pp.~84--97.

\bibitem{weber12}
{\sc A.~Weber, Y.~Prokazov, W.~Zuschratter, and M.~J.~B. Hauser}, {\em
  Desynchronisation of glycolytic oscillations in yeast cell populations}, PLoS
  One, 7 (2012), p.~e43276, \url{https://doi.org/10.1371/journal.pone.0043276}.

\bibitem{wechselberger_existence_2005}
{\sc M.~Wechselberger}, {\em Existence and bifurcation of canards in
  $\mathbb{R}^3$ in the case of a folded node}, {SIAM} Journal on Applied
  Dynamical Systems, 4 (2005), pp.~101--139,
  \url{https://doi.org/10.1137/030601995},
  \url{http://epubs.siam.org/doi/abs/10.1137/030601995} (accessed 2014-05-26).

\bibitem{yaru21}
{\sc L.~Yaru and L.~Shenquan}, {\em Characterizing mixed-mode oscillations
  shaped by canard and bifurcation structure in a three-dimensional cardiac
  cell model}, Nonlinear Dynamics,  (2021), pp.~1--22,
  \url{https://doi.org/10.1007/s11071-021-06255-z}.


	
\end{thebibliography}

\end{document}